\documentclass{amsart}
\usepackage[all]{xy}
\usepackage{amsmath,amsthm, graphicx}
\usepackage{color}
\usepackage{tikz-cd}
\usepackage[abs]{overpic}
 \usepackage{palatino}

\usepackage{hyperref}

\theoremstyle{plain}
\newtheorem{thm}{Theorem}[section]
\newtheorem{cor}[thm]{Corollary}
\newtheorem{conj}{Conjecture}[section]
\newtheorem{prop}[thm]{Proposition}
\newtheorem{lem}[thm]{Lemma}

\newtheorem*{mthm}{Theorem modulo above conjecture}

\theoremstyle{definition}
\newtheorem{remark}[thm]{Remark}
\newtheorem{example}[thm]{Example}
\newtheorem{ex}[thm]{Exercise}
\newtheorem{question}[thm]{Question}

\newcommand{\comment}[1]{}

\newcommand{\Q}{\ensuremath{\mathbb{Q}}}
\newcommand{\R}{\ensuremath{\mathbb{R}}}
\newcommand{\Z}{\ensuremath{\mathbb{Z}}}
\newcommand{\C}{\ensuremath{\mathbb{C}}}

\newcommand{\Chi}{\ensuremath{\mathcal{X}}}
\renewcommand{\sl}{\ensuremath{{\, sl}}}

\def\dfn#1{{\em #1}}

\title{Monoids in the mapping class group}

\author{John B. Etnyre}
\address{School of Mathematics \\ Georgia Institute of Technology}
\email{etnyre@math.gatech.edu}
\urladdr{\href{http://www.math.gatech.edu/~etnyre}{http://www.math.gatech.edu/\~{}etnyre}}

\author{Jeremy Van Horn-Morris}
\address{Department of Mathematics\\ The University of Arkansas} 
\email{jvhm@uark.edu}

\begin{document}

\begin{abstract}
In this article we survey, and make a few new observations about, the surprising connection between sub-monoids of the mapping class groups and interesting geometry and topology in low-dimensions. 
\end{abstract}

\maketitle

\section{Introduction}
Recall that a \dfn{monoid} is a set $M$ with an associative binary operation and a unit (or colloquially ``a group without inverses''). A typical example is the non-negative integers under addition. Below we describe various ways to construct monoids in groups. Specifically we discuss how to use generating sets and left-invariant orderings to create monoids in groups. We then apply these techniques to the braid groups and mapping class groups to generate several different sub-monoids. The surprising  thing is these algebraically defined monoids tend to have deep connections to topology and geometry! We then show how to use contact geometric ideas to construct more monoids in the mapping class groups that are algebraically not well understood at all. A better such understanding would most likely have important implications in contact geometry. Moreover previously intractable questions about various monoids in the braid group defined via generating sets (and long known to be related to algebraic geometry) have been answered using contact geometric techniques. 

It is this amazing connection between the algebra, on one hand, and geometry and topology, on the other, of monoids in the mapping class group that is the focus of this paper. We wish to highlight the connections that are known and point to many interesting open problems in the area. 

As we hope this survey paper to be highly accessible, in Section~\ref{bg} we recall (1) basic facts about the braid group and their relation to links in $S^3$, (2) a few definitions and facts about contact geometry in dimension 3, (3) the relation between braids and contact topology, (4) some results about open book decomposition and their relation to contact structures and generalized braids, and finally (5) some notation concerning the mapping class group of a surface. In the following section we discuss how to generate monoids via generating sets for a group and then discuss some classically known monoids in the braid group. Specifically we discuss positive, quasi-positive and strongly quasi-positive braids and their topological and geometric importance. 

In Section~\ref{ordermonoid} we show how to create monoids from left-invariant ordering and then discuss the famous ordering of the braid groups and more generally mappings class groups of surfaces with boundary. This will naturally lead to the notion of right-veering diffeomorphisms of surfaces and the right-veering monoid which has important connections to contact geometry. In Section~\ref{mvcg} we then indicate the surprising fact that once can construct many monids in the mapping class groups using ideas from contact geometry. The following section then discusses many interesting observations about these monoids and even more open questions about them. 

In the final two sections of the paper we show how to use contact geometry to construct monoids in the braid groups (and discuss many questions related to this) and how to use contact geometry to study the various classical ``positive" monoids in the braid group.

\subsection*{ Acknowledgments} We thanks Dan Margalit for helping clarify various points about mapping class groups, Mohan Bhupal and Burak Ozbagci for discussions about Milnor fillable contact structures used in Example~\ref{example}, and Liam Watson for helping develop some of the ideas in Section~\ref{bgmonoids}. We also thank Jamie Conway, Amey Koloti and Alan Diaz for help with Example~\ref{fdtcexex} below and other discussions that helped clarify various points in the paper. 
The first author gratefully acknowledges the support of the NSF grant DMS-1309073. The second author was partially supported by a grant from the Simons Foundation (279342).

\section{Background}\label{bg}

\subsection{The braid group}\label{braids}

Here we briefly discuss the braid group and will return to it from the perspective of mapping class groups later. For a more thorough introduction to the subject we recommend both Birman's classic book \cite{BIrman74} and the survey paper \cite{BirmanBrendle05}. 

We begin by fixing $n$ points $x_1,\ldots x_n$ on the $y$ axis inside the unit disk $D^2$ so that their $y$--coordinate increases with their index. Then recall that an \dfn{$n$--strand braid}, or \dfn{$n$--braid} for short, is an isotopy class of embeddings of $n$ intervals $[0,1]$ into $D^2\times [0,1]$ that is transverse to each disk $D^2\times \{t\}$ and intersects $D^2\times\{0\}$ and $D^2\times \{1\}$ in the points $\{x_1,\ldots, x_n\}$. It is clear that given two $n$-braids $w_1$ and $w_2$ we can reparameterize the first to lie in $D^2\times[0,1/2]$ and the second to lie in $D^2\times [1/2,1]$ and then concatenate them to get a new braid $w_1w_2$. Thus the set of $n$--braids, which we denote by $B(n)$, has a multiplicative structure. One may easily see that this gives $B(n)$ the structure of a group. This is called the braid group.

\begin{figure}[htb]
\begin{overpic}
{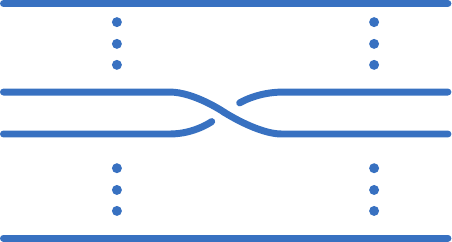}
\put(-8, -2){$1$}
\put(-8,29){$i$}
\put(-25, 41){$i+1$}
\put(-8, 67){$n$}
\end{overpic}
\caption{The standard generator $\sigma_i$ of the braid group $B(n)$.}
\label{fig:sigman}
\end{figure}
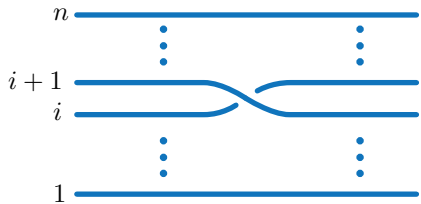
There is a simple finite presentation for the braid group called the Artin presentation, \cite{Artin25}. The generators are the braids $\sigma_i$ depicted in Figure~\ref{fig:sigman} for $i=1,\ldots, n-1$ and the relations are 
\[
\sigma_i\sigma_j=\sigma_j\sigma_i \quad \text{if } |i-j|>1
\]
and
\[
\sigma_i\sigma_{i+1}\sigma_i=\sigma_{i+1}\sigma_i\sigma_{i+1}.
\]

Notice that given a braid $w$ we can consider the image of $B$ in $D^2\times [0,1]$ with $D^2\times\{0\}$ glued to $D^2\times \{1\}$ by the identity. This gives a link in $D^2\times S^1$ and if we identify $D^2\times S^1$ with the neighborhood of the unknot in $\R^3$ (taking the product framing to the zero framing on the unknot) then the image of $w$ will be a link $\overline{w}$ in $\R^3$ called the closure of $w$. Notice that if one uses cylindrical coordinates on $\R^3$ then the identification of $D^2\times S^1$ with a neighborhood of the unknot can be done in such a way that the $S^1$ factor is the $\theta$ coordinate. Thus we see that the $\theta$--coordinate is monotonically increasing as we traverse any component of $\overline{w}$. Similarly if $L$ is any link in $\R^3$ that is disjoint from the $z$--axis and has $\theta$--coordinate monotonically increasing then we say that $L$ is \dfn{braided} about the $z$--axis. Notice that in this case we can isotope $L$, though links braided about the $z$--axis, so that it lies in a neighborhood $D^2\times S^1$ of the unknot in the $z=0$ plane. Choosing any $\theta_0$ $L$ will intersect $D^2\times \{\theta_0\}$ in some number of points, say $n$. Now we can further isotope $L$ in $D^2\times S^1$ so that it intersects $D^2\times \{\theta_0\}$ in the points $x_0,\ldots x_n$.  Thus cutting $D^2\times S^1$ along $D^2\times \{\theta_0\}$  will result in $D^2\times [0,1]$ and $L$ will become a braid $w$. It is clear that $L$ is isotopic to the closure $\overline{w}$ of the braid $w$. It is quite useful to know that all links can be so represented. 
\begin{thm}[Alexander 1923, \cite{Alexander23}]
Given any link $L$ in $\R^3$ there is some natural number $n$ and $n$--braid $w$ such that $L$ is isotopic to the closure or $w$. 
\end{thm}
It should be clear that the braid in Alexander's theorem is not unique. For example given a $n$--braid $w$ we can form an $(n+1)$--braid $w^\pm$ called the \dfn{positive/negative stabilization} of $w$ by adding a trivial $(n+1)^\text{st}$ strand and then multiplying $w$ by the generator $\sigma_n^{\pm}$. One easily checks that the closures of $w$ and $w^\pm$ are isotopic links. Similarly given $w$ one can conjugate $w$ by another $n$--braid $b$ to get a braid with the same closure: $\overline{w}$ is isotopic to $\overline{bwb^{-1}}$. It turns out these two procedures are the only way to get braids with the same closure. 
\begin{thm}[Markov 1935, \cite{Markov35}]
Two braids $w$ and $w'$ have isotopic closures if and only if $w$ and $w'$ are related by a sequence of stabilizations, destabilizations and conjugations. 
\end{thm}
Thus we see that the study of knots can be encoded in the study of the braid groups. We also remark that all the above statements hold for links in $S^3$ as well as $\R^3$ and we will switch between these two settings when convenient.

\subsection{Contact structures}
A contact structure $\xi$ on a 3--manifold $M$ is a 2--dimensional sub-bundle of the tangent bundle of $M$ that is not tangent to any surfaces along an open subset of the surface. This is most conveniently expressed in terms of a (locally defined) 1--form $\alpha$ such that $\xi=\ker \alpha$ and $\alpha\wedge d\alpha\not=0$ at any point. The assumption that we can choose $\alpha$ globally is equivalent to $\xi$ being orientable which we will assume throughout this paper.  The canonically example of a contact structure is $\xi_{std}=\ker (dz+ r^2d\theta)$ on $\R^3$, see Figure~\ref{firstex}. 
\begin{figure}[htb]
\begin{overpic}
{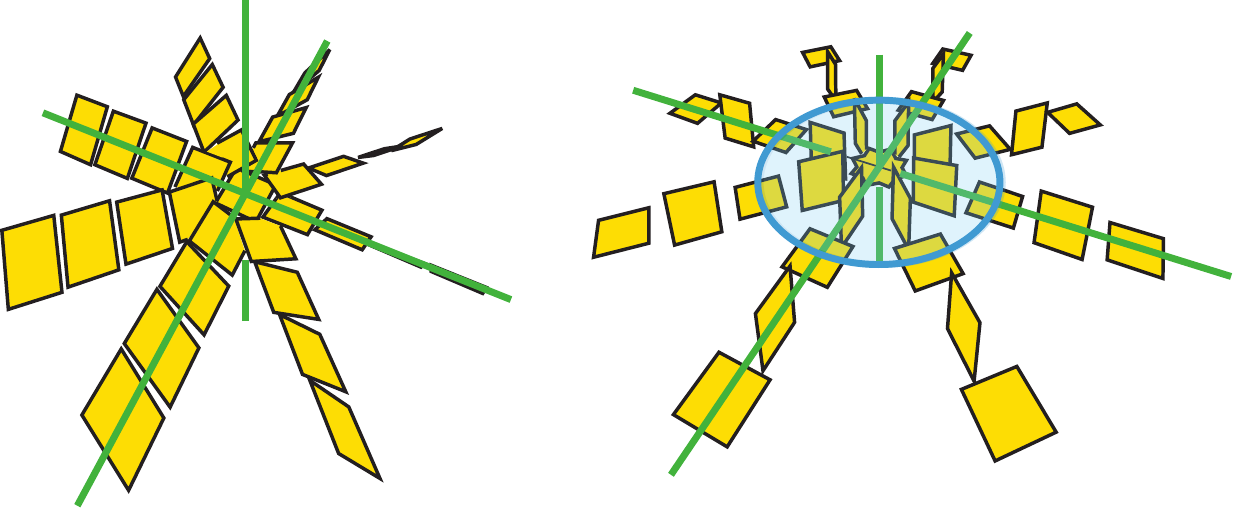}
\end{overpic}
\caption{The contact structure $\xi_{std}$ and $\xi_{ot}$ with the overtwisted disk indicated on the right. (Figure courtesy of S.\ Sch\"onenberger.)}
\label{firstex}
\end{figure}
A well known theorem of Darboux, see \cite{Geiges08}, says that any contact structure on a 3--manifold is locally equivalent to $(\R^3, \xi_{std})$. So we could have defined a contact structure to be a 2--dimensional sub-bundle $\xi$ of the tangent bundle that is locally modeled on $(\R^3, \xi_{std})$. Compactifying $\R^3$ to $S^3$ we get an induced contact structure $\xi_{std}$ on $S^3$ that can also be seen as the set of complex tangencies to $S^3$ when $S^3$ is thought of as the unit sphere in $\C^2$. 

In dimension 3 contact structures fall into one of two categories: tight and overtwisted. We say a contact structure $\xi$ on $M$ is overtwisted if there is a disk $D$ embedded in $M$ such that $D$ is tangent to $\xi$ along its boundary and at one point on the interior. Such a disk is called an overtwisted disk. If no such disk exists then the contact structure is called tight. An example of an overtwisted disk can be seen in $\R^3$ with the contact structure $\xi_{ot}=\ker (\cos r\, dz + r\sin r\, d\theta)$ as the disk of radius $\pi$ in the plane $\{z=0\}$. See Figure~\ref{firstex}.

We will return to the tight verses overtwisted dichotomy shortly, but first recall a few other basic definitions and facts about contact structures. In the study of contact structures it is important to consider knots adapted to a contact structure in various ways. We call a knot $K$ in a contact manifold $(M,\xi)$ \dfn{Legendrian} if $K$ is always tangent to $\xi$ and we call it \dfn{transverse} if $T_xK$ is transverse to $\xi_x$ in $T_xM$ for all $x\in K$. When considering isotopies of such knots we always consider isotopies through knots with the same property. It turns out that each topological knot type can be realized by many different Legendrian and transverse knots (that is a topological knot can be topologically isotoped to be a Legendrian or transverse knot in many different ways). For more on Legendrian and transverse knots see \cite{Etnyre05}, but here we just recall some basic invariants of such knots. First there are two classical invariants of a Legendrian knot $K$. There is the framing $fr(K,\xi)$ that $\xi$ gives to $K$. If $K$ is null-homologous it also has a framing coming from a Seifert surface $\Sigma$, that is a surface with boundary $K$. In this case the difference between $fr(K,\xi)$ and the Seifert framing is called the \dfn{Thurston-Bennequin invariant} of $K$ and is denoted $tb(K)$. It is easy to show that $\xi$ restricted to $\Sigma$ is trivial, thus we can pick a non-zero vector field $v$ in $\xi$ along $\Sigma$. Now if $K$ is oriented then we can also take an oriented tangent vector field $w$ along $K$. The \dfn{rotation number} of $K$ is the rotation of $w$ along $K$ with respect to $v$. We denote this number by $r(K)$. One may easily check that Legendrian knots that are Legendrian isotopic have the same Thurston-Bennequin invariants and rotation numbers. 

Later we will be interested in surgeries on Legendrian knots. Specifically, given a Legendrian knot $L$ in a contact manifold $(M,\xi)$ then $L$ has a neighborhood that is contactomorphic to a product neighborhood of $S^1\times \{(0,0)\}$ in $S^1\times \R^2$ with the contact structure $dz-y\, d\theta$ where $(y,z)$ are Euclidean coordinates on $\R^2$ and $\theta$ is the angular coordinate on $S^1$. If one removes this neighborhood and glues back in a solid torus to perform $\pm 1$--surgery, with respect to the framing of $L$ given by the contact planes, then there is a unique way to extend the contact structure on the complement of the neighborhood over the surgery torus so that it is tight on the surgery torus. The resulting contact manifold is called $\pm 1$--contact surgery on $(M\xi)$ along $L$. Also, $-1$--contact surgery is called Legendrian surgery. It is known that all contact 3--manifolds can be obtained from the standard contact structure on $S^3$ by a sequence of $\pm 1$--contact surgeries.

Now given a transverse knot $K$ that is the boundary of a surface $\Sigma$ we can again trivialize $\xi|_\Sigma$ and use a non-zero section of $\xi|_\Sigma$ to push off a copy $K'$ of $K$. Then the \dfn{self-linking number} of $K$ is just the linking number of $K$ and $K'$ (that is the intersection number of $K'$ and $\Sigma$). We denote this number by $sl(K)$ and can easily see that it is an invariant of the transverse isotopy class of $K$. 

We finish this section by briefly reviewing various types of fillings. For a more leisurely discussion see \cite{Etnyre98}. We say a contact manifold $(M,\xi)$ is \dfn{strongly symplectically filled} by the symplectic manifold $(X,\omega)$ if $X$ is a compact manifold with $\partial X=M$ and there is a vector field $v$ on $X$ defined near the boundary of $X$ so that $v$ is transverse to $\partial X$, the flow of $v$ dilates $\omega$ (that is $L_v\omega=\omega$ where $L$ stands for the Lie derivative), and $\alpha = (\iota_v\omega)|_M$ is a contact form for $\xi$. We say that $(X,\omega)$ is a \dfn{weak symplectic filling} of $(M,\xi)$ if again $X$ is compact, $\partial X=M$ and $\omega$ is a symplectic form when restricted to $\xi$. One can easily see that if $(X,\omega)$ is a strong symplectic filling of $(M,\xi)$ then it is also a weak symplectic filling. (Sometimes the word symplectic is left out when talking about fillings.) One key fact about fillings is given in the following results.
\begin{thm}[Gromov 1985, \cite{Gromov85}; Eliashberg 1990, \cite{Eliashberg90b}]
Any weakly symplectically fillable contact structure is tight. 
\end{thm}

We have one final type of filling of a contact structure. A complex manifold $(X,J)$ is called a \dfn{Stein manifold} if it admits a proper function $\phi:X\to \R$ that is bounded below and plurisubharmonic. By plurisubharmonic we mean that $-d(J^* d \phi)$ is a symplectic form on $X$. A Stein domain is a regular sub-level set of $\phi$ and we say that $(M,\xi)$ is Stein fillable if it is the boundary of a Stein domain defined by $\phi:X\to\R$ and $-J^*d\phi$ restricted to $M$ is a contact form for $\xi$. One can easily check that if a contact manifold is Stein fillable then it is also strongly and weakly symplectically fillable. 

We end by noting that Legendrian surgery preserves all forms of fillability. That is if $(M,\xi)$ is fillable in some sense and $(M',\xi')$ is obtained from it by Legendrian surgery on some link then $(M',\xi')$ is also fillable in the same sense, \cite{Eliashberg90a, EtnyreHonda02b, Weinstein91}. 

\subsection{Braids and contact topology}
The birth of modern contact topology could very well be Bennequin's seminal paper showing that there were at least two distinct contact structures on $\R^3$. He did this by showing that all transverse knots in the standard contact structure $\xi_{std}=\ker(dz+r^2\, d\theta)$ satisfy the Bennequin inequality but transverse knots in $\xi_{ot}=\ker(\cos r\, dz + r\sin r\, d\theta)$ do not. More specifically he proved the following result. 
\begin{thm}[Bennequin 1983, \cite{Bennequin83}]\label{bennequin}
If $T$ is any knot transverse to the standard contact structure $\xi_{std}$ on $\R^3$ then 
\begin{equation}
sl(T)\leq -\chi(\Sigma),
\end{equation}
where $\Sigma$ is any Seifert surface for $T$. 
\end{thm}
As the genus of a knot is determined by its Euler characteristic this gives a lower bound on the genus of a knot. So not only did this theorem indicate that there is more than one contact structure on $\R^3$  it also shows that contact structures can give interesting purely topological information. The Bennequin inequality had taken on a central role in contact topology thanks in large part to Eliashberg connecting it with tightness.
\begin{thm}[Eliashberg, 1992 \cite{Eliashberg92a}]\label{eliashberg-tight}
Let $\xi$ be a contact structure on a 3--manifold $M$. Then the following are equivalent:
\begin{enumerate}
\item\label{it1} The contact structure $\xi$ is tight ({\em ie} contains no overtwisted disks).
\item\label{it1a} The contact structure $\xi$ contains no embedded disks with Legendrian boundary and contact framing 0.
\item\label{it2} All transverse knots in $(M,\xi)$ satisfy the Bennequin bound. 
\item\label{it3} There is a topological knot type such that any transverse knot in that topological knot type satisfies the Bennequin bound. 
\item\label{it4} There is a topological knot type such that any transverse knot in that topological knot type satisfies {\em any} upper bound bound on their self-linking numbers. 
\end{enumerate}
\end{thm}
The implications that both Item~\eqref{it2} and Item~\eqref{it1a} each imply \eqref{it1} is clear. The reverse implications were established by Eliashberg in \cite{Eliashberg92a} by studying characteristic foliations and careful application of the Giroux cancellation lemma (see \cite{Etnyre03} for an exposition of this). Clearly Item~\eqref{it2} implies Item~\eqref{it3} which in turn implies Item~\eqref{it4}. The fact that Item~\eqref{it4} implies Item~\eqref{it1} is easily established by showing that if $\xi$ is not tight then there is no bound for any knot type by noting that an overtwisted contact structure has a transverse unknot with self-linking number 1 (just take the boundary of a slight enlargement of an overtwisted disk) and that the self-linking number is additive plus one under connected sum ({\em ie} $sl(K\# K')=sl(K)+sl(K')+1$). In particular while clever the proof of Theorem~\ref{eliashberg-tight} is elementary. 

The proof of Theorem~\ref{bennequin} is a beautiful application of braid theoretic techniques, see \cite{Bennequin83} for details. Here we will just indicate the connection between braids and contact geometry. 

We first observe that a closed braid in $\R^3$ is naturally a transverse knot in $\xi_{std}$. To see this notice that when the $r$-coordinate is large the contact planes $\xi_{std}$ are almost tangent to the half-spaces $\{\theta=c\}$. As, by definition, a closed braid transversely intersects these half spaces if we isotope a closed braid so that its $r$-coordinate is sufficiently large we see that it is transverse to $\xi_{std}$ too. Moreover given any two representatives of the same closed braids with large enough $r$-coordinates they will clearly be transversely isotopic. Thus we see that braids naturally give us transverse knots in $(\R^3, \xi_{std})$. Bennequin observed that all transverse knots can be so expressed. 
\begin{thm}[Bennequin, 1983 \cite{Bennequin83}]
Let $T$ be a transverse knot in $(\R^3,\xi_{std})$. Then $T$ can be isotoped through transverse knots to be the closure of a braid. 
\end{thm}
We will not give a detailed proof hear, but simply notice that standard braiding procedures for knots (that is techniques to proof Alexander's theorem) can be easily adapted to transverse knots. The reader is encouraged to prove this themselves. 

We briefly recall another  braid theory result that generalize from topological knots to transverse knots. Specifically the Markov Theorem has the following transverse analog. 
\begin{thm}[Orevkov and Shevchishin, 2003 \cite{OrevkovShevchishin03}; Wrinkle, 2002 \cite{Wrinkle02}]\label{posstab}
The closure of two braids represent transversely isotopic transverse knot in $(R^3,\xi_{std})$ if and only if they are related by (1) conjugation in the braid group ({\em ie} isotopy as close braids) and (2) {\em positive} Markov stabilizations and destabilizations. 
\end{thm}
The two theorems above say that one can study transverse knots in $(\R^3,\xi_{std})$ by just studying braids and that in this respect the only difference between transverse and topological knots is {\em negative} Markov stabilizations. In particular one can compute the self-linking number of a transverse knot in term so a braid representing it. 
\begin{lem}[Bennequin, 1983 \cite{Bennequin83}]\label{braidsl}
Let $w$ be an element in the $n$-strand braid group $B(n)$. The transverse knot represented by the closure $\overline{w}$ of $w$ has self-linking number 
\[
sl(\overline{w})=\text{writhe}(w)-n,
\]
where $\text{writhe}(w)$ is the writhe of the natural digram for the closure of $w$ and can be computed as the exponent sum of the word $w$. That is if $w=\sigma^{\epsilon_1}_{i_1}\ldots \sigma^{\epsilon_k}_{i_k}$ in the braid group $B(n)$, where each $\epsilon_i$ is either $1$ or $-1$,  then $\text{writhe}(w)=\sum \epsilon_i$.
\end{lem}

\subsection{Open book decompositions and braids}
Given a surface $\Sigma$ with boundary and a diffeomorphism $\phi:\Sigma\to \Sigma$ that restricts to be the identity map in a neighborhood of $\partial \Sigma$ one can form the mapping torus $T_\phi$ of $\phi$, that is the front and back of $\Sigma \times [0,1]$ are identified with $\phi$:
\[
T_\phi=\Sigma\times[0,1]/(p,1)\sim(\phi(p),0).
\]
Notice that $\partial T_\phi$ is a union of copies of $S^1\times S^1$ (the product structure coming from $S^1\times[0,1]/\sim$ where $S^1$ is a boundary component of $\Sigma$). We can now glue a copy of $S^1\times D^2$ to each boundary component of $T_\phi$ so that $S^1\times \{pt\}$ is glued to $S^1\times \{t\}$ and $\{pt\}\times \partial D^2$ is glued to $\{\theta\}\times S^1$. See Figure~\ref{openbook}.
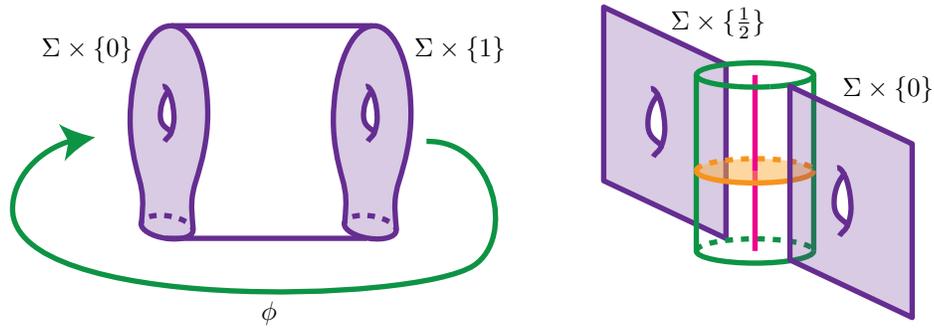
\begin{figure}[htb]
\begin{overpic}
{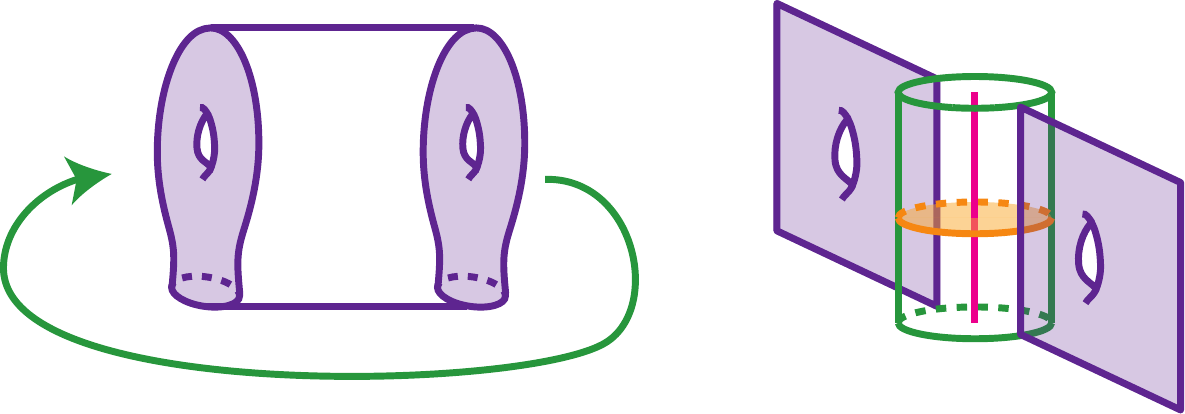}
\put(12,100){$\Sigma\times\{0\}$}
\put(153,100){$\Sigma\times\{1\}$}
\put(95,0){$\phi$}
\put(250,110){$\Sigma\times\{\frac 12\}$}
\put(315,85){$\Sigma\times\{0\}$}
\end{overpic}
\caption{Constructing an open book decomposition. On the left we see a the mapping torus $\Sigma\times[0,1]$ with $\Sigma\times\{1\}$ glued to $\Sigma\times\{0\}$ by $\phi$. On the right we see the solid torus (green) glued to the mapping cylinder. The red circle is the binding (core of the solid torus).}
\label{openbook}
\end{figure}
This gives a closed 3--manifold $M_\phi$. We say that $(\Sigma,\phi)$ is an \dfn{open book decomposition} for $M$ if $M_\phi$ is diffeomorphic to $M$ (technically this diffeomorphism should be part of the structure but is usually left implicit in discussions of open books). We call $\Sigma$ a \dfn{page} of the open book and $\phi$ the \dfn{monodromy}.

Notice that the cores of added solid tori form a link $B$ in $M_\phi\cong M$ and the complement of $B$ is diffeomorphic to the interior of $T_\phi$ so there is a fibration $\phi:(M-B)\to S^1$ whose fibers are the interior of Seifert surfaces for $B$. The link $B$ is called the \dfn{binding} of the open book and the closures of the fibers of $\pi$ are called the \dfn{pages}. The pair $(B,\pi)$ can easily be seen to determine $(\Sigma, \phi)$ ($\phi$ only up to conjugation) and is also called an open book decomposition for $M$. The two definitions of open book are often used interchangeably but one should note they are not quite the same. In any event, we will mainly consider open books in terms of pages and monodromies. 
\begin{example}\label{ex}
As a simple example we consider the open book with page a disk $D^2$ and monodromy $\phi$ the identity map. Clearly $T_\phi$ is simply $D^2\times S^1$ and so when a solid torus is glued to this as above we get $M_\phi$ diffeomorphic to $S^3$. Moreover the core of the added solid torus is $B$ the unknot in $S^3$ so the binding of this open book is the unknot. 
\end{example}

Generalizing the notion of braid we can consider an $n$--strand braid in $\Sigma\times [0,1]$. That is fix $n$ points on $\Sigma$, say $\{x_1,\ldots, x_n\}$. Then an $n$--strand braid in $\Sigma\times [0,1]$, or $n$--braid for short, is an isotopy class of embeddings of $n$ intervals $[0,1]$ into $\Sigma\times [0,1]$ that is transverse to each disk $\Sigma\times \{t\}$ and intersects $\Sigma\times\{0\}$ and $\Sigma\times \{1\}$ in the points $\{x_1,\ldots, x_n\}$. Just as for ordinary braids one can multiply two braids and see that we get a group $B(n,\Sigma)$. Moreover it is clear that one can take the closure of a braid $b$ in $B(n,\Sigma)$ to get a closed link $\overline{b}\in M_\phi$ for any $\phi$ that preserves the points $\{x_1,\ldots, x_n\}$ (and as we can assume these points are in a neighborhood of $\partial \Sigma$ where $\phi$ is the identity one can ignore this last condition). There is an obvious notion of positive and negative stabilization as in the standard braid group where an extra ``stand" is added (that is $x_{n+1}\times [0,1]$ where $x_{n+1}$ is a point close to $x_n$) and then the original braid is multiplied by the braid with a single positive or negative half-twist is added between the $x_n$ and $x_{n+1}$ strands. 
\begin{example}
Returning to the open book for $S^3$ in Example~\ref{ex} above we see that a braid in the just defined sense is exactly a braid in the original sense. That is $B(n,D^2)$ is exactly $B(n)$ and the just defined stabilization operations give ordinary braid stabilization. 
\end{example}
We now have the generalization of Alexander's and Markov's theorems. 
\begin{thm}[Skora 1992, \cite{Skora92}; Sundheim 1993, \cite{Sundheim93}]
Given an open book $(\Sigma,\phi)$ any link in $M_\phi$ can be represented as the closure of a braid of some index. Moreover if two braids represent the same link then they are related by (1) isotopies as closed braids (that is isotopies through links that are transverse to the pages of the open book and disjoint form the binding) and (2) positive and negative stabilization. 
\end{thm}
Notice that original version of Markov's theorem two braids gave the same knot if they were related by conjugation and stabilization and in the above theorem they need to be related by ``braid isotopy" and stabilization. One may easily check that for braids in the open book $(D^2, id_{D^2})$ braid isotopy is exactly conjugation in the braid group. There is a similar statement for general open books but the conjugation is ``twisted" by the monodromy map. The reader is encouraged to work out explicitly what this means though it will not be relevant for what follows.

\subsection{Contact structures and the Giroux correspondence}\label{girouxcor}
In 1975 Thurston and Winkelnkemper, \cite{ThurstonWinkelnkemper75},  showed how to associate a contact structure $\xi_{\phi}$ on the manifold $M_\phi$ to an open book $(\Sigma,\phi)$. Later in 2002 Giroux \cite{Giroux02} made the following definition: a contact structure $\xi$ on $M$ is supported by the open book $(B,\pi)$ if there is a 1--form $\alpha$ with $\xi=\ker \alpha$, $\alpha$ is non-zero on $B$, $d\alpha$ is non-zero on the pages of the open books and the orientation induced on the pages by $d\alpha$ and on $B$ by $\alpha$ agree. Intuitively this means that the binding is transverse to the contact structure and the contact planes, away from the binding can be isotoped to be arbitrarily close to the tangent planes to the pages. Giroux then noted that with care the Thurston and Winkelnkemper construction gave a contact structure supported by the given open book and that there is a unique contact structure compatible with a given open book. Giroux then made the amazing discovery that every contact structure is supported by some open book! He actually showed more but first we need a definition.

Given an open book $(\Sigma,\phi)$ and a properly embedded arc $c$ in $\Sigma$ then the \dfn{positive stabilization} of $(\Sigma,\phi)$ is the open book $(\Sigma',\phi')$, with page $\Sigma'$ obtained from $\Sigma$ by attaching an (oriented) 1-handle along $\partial c$ and monodromy $\phi'=\tau_\gamma\circ \phi$ where $\tau_\gamma$ is a Dehn twist (see below) along $\gamma$ which is the closed curve formed as the union of $c$ and the core of the added 1-handle. One can similarly define a negative stabilization in the same way but using $\tau_\gamma^{-1}$ instead of $\tau_\gamma$. One can show that the manifolds $M_\phi$ and $M_{\phi'}$ associated to an open book and its stabilization (either positive or negative) are diffeomorphic. It is an interesting exercise to work out the definition of stabilization in terms of the binding and fibration: $(B,\pi)$. Giroux observed that the contact structure supported by an open book and its positive stabilization are isotopic. He moreover proved the following one-to-one correspondence. 
\begin{thm}[Giroux 2002, \cite{Giroux02}]
Let $M$ be a closed oriented manifold. Then there is a one-to-one correspondence between
\[
\{\text{oriented contact structures up to isotopy}\}
\]
\center{and} 
\[
\{\text{open book decompositions } (B,\pi) \text{ up to isotopy and positive stabilization}\}.
\]
\end{thm}
This amazing result has been a cornerstone of contact geometry ever sense and will feature prominently in our discussion of monoids in mapping class groups. 

Just as we did for ordinary braids we can associate transverse links to braids in open books. Specifically, given a braid $b$ in $B(n,\Sigma)$ we can consider its closure $\overline{b}$ in $M_\phi$ and note that it is transverse to the pages. Since, as discussed above, we may assume the contact planes are very close to the tangent planes to the pages we see that $\overline{b}$ is a transverse link. We say a transverse link coming from the closure of a braid is \dfn{braided about the open book}. We end this subsection by noting for future use that Bennequin, Orevkov-Shevchishin, and Wrinkles theorems have been generalized to transverse links in all open books.
\begin{thm}[Pavelescu, 2008 \cite{Pavelescu08}]
Let $\xi$ be a contact structure on a 3-manifold $M$ that is supported by an open book $(B,\pi)$. 
\begin{enumerate}
\item Any transverse knot $T$ in $(M,\xi)$ can be braided about $(B,\pi)$.
\item Two braids are isotopic through transverse knots if and only if they are related by braid isotopy and positive Markov moves. 
\end{enumerate}
\end{thm}

\subsection{Mapping class groups}
We denote a surface of genus $g$ with $k$ boundary components and $n$ marked points by $S_n^{g,k}$. We will frequently leave out superscripts or subscripts if they are clear from context or unimportant. The \dfn{mapping class group} of $S_n^{g,k}$ is the group of diffeomorphisms of $S_n^{g,k}$ preserve the marked points set-wise modulo isotopies through such maps. We denote this group
\[
Mod(S_n^{g,k}).
\]
If the diffeomorphisms are the identity on the boundary of $S_n^{g,k}$ the the group of isotopy classes is denoted
\[
Mod(S_n^{g,k}, \partial S_n^{g,k}).
\]
We will discuss the mapping class group more in the following sections.

\section{Monoids via generating sets}\label{monoidsgs}
Given any group $G$ and a subset of that group $A$ that generates $G$ (that is every element of $G$ can be written as a product if the elements of $A$ and their inverses) then we can get a sub-monoid of $G$ by  
\[
M(A)=\{\text{all words in non-negative powers of the elements of } A\}.
\]
In the following two subsections we consider monoids in the mapping class group defined using generating sets and see how they are related to interesting geometric and topological properties.
\subsection{Dehn twists}\label{dt}
Given an oriented surface $S$, any embedded closed curve $\gamma$ in $S$ has a neighborhood diffeomorphic to $S^1\times [0,1]$. The (positive) Dehn twist along $\gamma$ is the a diffeomorphism 
\[
\tau_\gamma:S\to S
\]
that is the identity map outside the neighborhood of $\gamma$ and equal to the map $(\theta,t)\to (\theta - 2\pi f(t), t)$ for points in the neighborhood $S^1\times [0,1]$, where $f:[0,1]\to [0,1]$ is a non-decreasing function that is identically 0 near 0 and identically 1 near 1. See Figure~\ref{posdehntwist}.
\begin{figure}[htb]
\begin{overpic}
{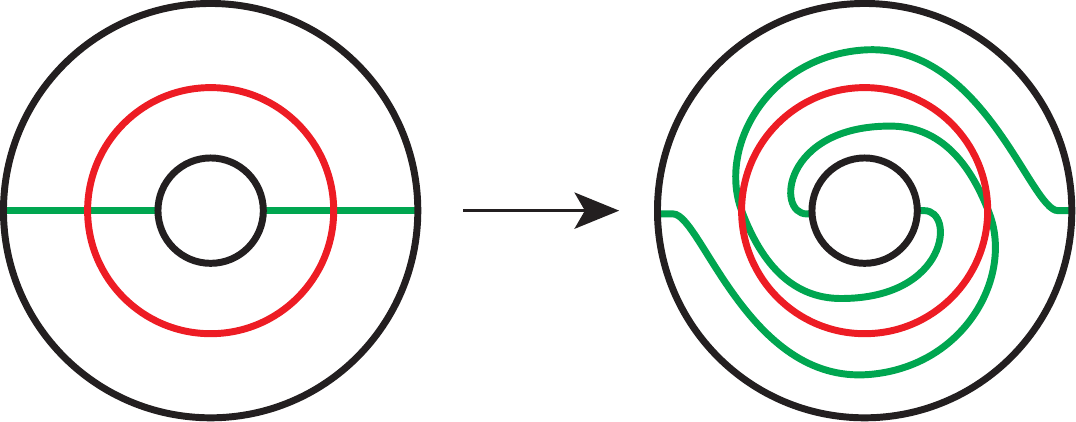}
\put(150, 70){$\tau_\gamma$}
\put(90,90){$\gamma$}
\end{overpic}
\caption{A positive Dehn twist about $\gamma$.}
\label{posdehntwist}
\end{figure}
This map is clearly smooth and it is a standard fact that the isotopy class of $\tau_\gamma$ is independent the specific choice of neighborhood of $\gamma$ and choice of $f$, and only depends on $\gamma$ up to isotopy. 
\begin{thm}[Dehn, 1920's \cite{Dehn87}; Lickorish, 1964 \cite{Lickorish64}]
The set of positive Dehn twists along curves in $S$ generate the mapping class group $Mod(S)$. 
\end{thm}
We denote by $Dehn^+(S)$ the monoid in $Mod(S, \partial S)$ generated by positive Dehn twist about curves in $S$. Using relations in the mapping class group, see for example \cite{FarbMargalit12}, it is easy to see the following.
\begin{thm}
If $S$ is a compact surface without boundary, then $Mod(S)=Dehn^+(S)$. 
\end{thm}
However if $S$ is a surface with boundary $Dehn^+(S)$ is {\bf not} a group and is a proper sub-monoid of $Mod(S, \partial S)$. We will discuss this monoid further in a subsequent subsection. We point out that this is a non-trivial fact that can easily be seen from contact geometry using the Giroux correspondence. (There are other proofs of this as well, for example one can use orderings as discussed in Section~\ref{ordermonoid}, and the contact geometric proof of this is essentially a repackaging of the ordering proof.)

\subsection{Half twists along arcs}\label{ht} 
Suppose that $S_n$ is an oriented surface with $n$ marked points $\{x_1,\ldots x_n\}$. Recall a diffeomorphism of $S_n$ is allowed to permute the punctures. Consider an arc $\alpha$ whose interior is embedded in $S_n$ in the complement of the marked points and who end points map to distinct marked points. There is a neighborhood of $\alpha$ in $S$ that contains only the marked points $\partial \alpha$ and is orientation preserving diffeomorphic to the disk disk of radius 2 about the origin in $\R^2$ by a diffeomorphism taking $\alpha$ to $\{(x,y): x=0 \text{ and } |y|\leq 1\}$. Let $f:[0, 2]\to [0,1]$ be a function that is equal to 0 for near $r= 0$, equal to $1$ near $r=2$, equal to $1/2$ for $r=1$ and is non-decreasing. We can now define a diffeomorphism $h_\alpha:S_n\to S_n$ by a $(r,\theta)\mapsto (r, \theta- f(r)\pi)$ on the disk of radius 2 (using polar coordinates) and equal to the identity outside of the neighborhood of $\alpha$. This diffeomorphism exchanges the two marked points involved. We call $h_\alpha$ a {\em (positive) half-twist along $\alpha$}.

It is well known that $Mod(S_n,\partial S_n)$ is generated by positive Dehn twists about embedded closed curves and positive half-twists along arcs. We explore the special case when the surface $D_n$ is a disk with $n$ marked points. In this case $Mod(D_n,\partial D_n)$ is better known as the braid group $B(n)$. 
To see this notice that given an element $[\phi]\in Mod(D_n, \partial D_n)$ if we think of $\phi$ as a diffeomorphism of $D^2$ then it is isotopic to the identity (since all diffeomorphisms of $D^2$ that fix the boundary are). Let $\Phi:D^2\times [0,1]\to D^2$ be this isotopy. 
Then the trace of the isotopy is an $n$--braid in $D\times [0,1]$:
\[
\text{image}\left\{\Phi:\left(\{x_1,\ldots, x_n\}\times [0,1]\right)\to D\times [0,1]\right\}.
\]
See Figure~\ref{twisttogen}. Using the notation in Section~\ref{braids} we can choose arcs $\alpha_i$ on the $y$-axis connecting $x_i$ to $x_{i+1}$ in $D^2$. Notice that the positive half-twist along $\alpha_i$ corresponds to the braid $\sigma_i$. 
\begin{figure}[htb]
\begin{overpic}
{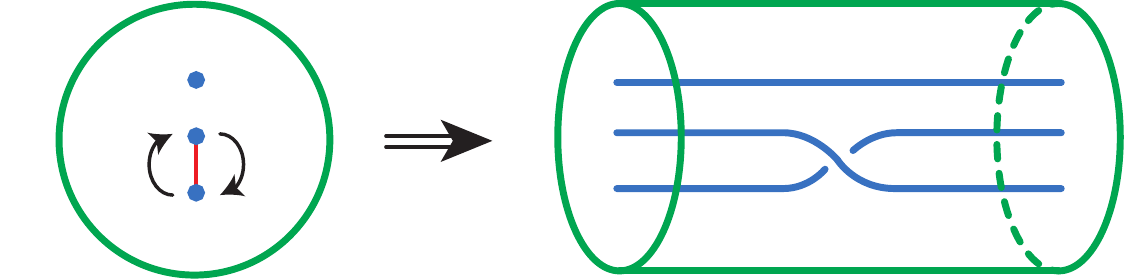}
\put(46,30){$\sigma_1$}

\end{overpic}
\caption{On the left is a positive half twist about the red arc. On the right is the corresponding braid $\sigma_1$.}
\label{twisttogen}
\end{figure}

We will use various generating sets to discuss three notations of positivity in the braid group. 

\subsubsection{Standard generators}
First consider the standard generators $\sigma_1,\ldots, \sigma_{n-1}$ for the braid group $B(n)$ given in Section~\ref{braids}. We call a braid \dfn{positive} if it can be written as a positive word in the $\sigma_i$. Thus we get the \dfn{positive monoid} $P(n)$ in $B(n)$. 

Before moving on we make a few quick observations about closures of positive braids. First notice that one can easily build a Seifert surface for a knot given as the closure of a braid word. Specifically given a braid word $w=\sigma^{\epsilon_1}_{i_1}\ldots \sigma^{\epsilon_k}_{i_k}$, where each $\epsilon_i$ is either $1$ or $-1$, in the braid group $B(n)$ we can take $n$ parallel disks $D_1,\ldots, D_n$ such that each has constant $z$-coordinate and the $z$-coordinate is increasing with the index on $D_i$. See Figure~\ref{BuildSS}. 
\begin{figure}[htb]
\begin{overpic}
{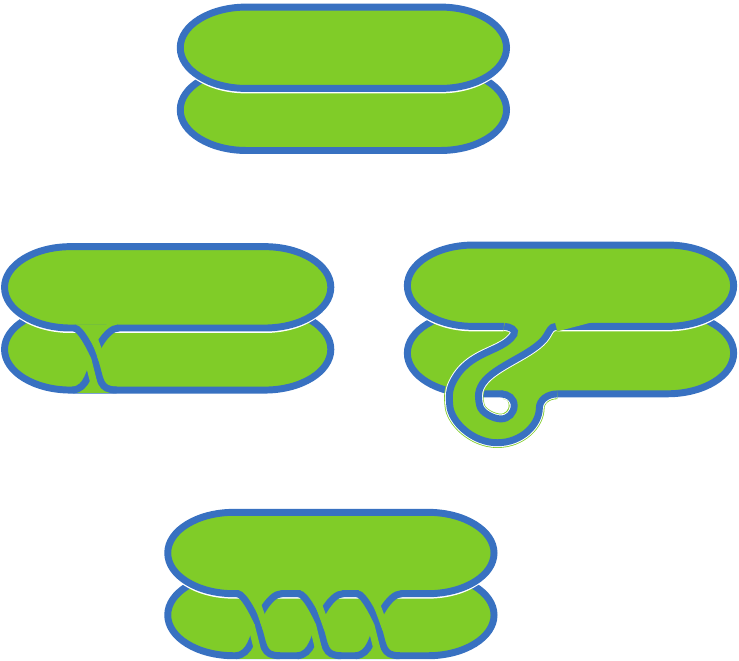}
\put(35,155){$D_1$}
\put(35,175){$D_2$}
\end{overpic}
\caption{On the top are two disks $D_1$ and $D_2$ whose boundary is the closure of the trivial 2--braid. In the middle, left a half twisted band is added between the two disks. The resulting surface has boundary the closure of the braid $\sigma_1$. In the middle, right is an alternate view of the half twisted band that will be used later. On the bottom is a surface with boundary the closure of the braid $\sigma_1\sigma_1\sigma_1$.}
\label{BuildSS}
\end{figure}
Notice that the boundary of these disks gives the closure of the trivial $n$--braid. Now for each letter in word $w$ we add a half twisted band connecting the corresponding disks. Again, see Figure~\ref{BuildSS}. The resulting surface $\Sigma_w$ clearly has boundary the closure of $w$. 
\begin{prop}\label{postive-min-genus}
If $w$ is a positive braid word then the Seifert surface $\Sigma_w$ constructed above is a minimal genus Seifert surface for the closure of $w$. 
\end{prop}
\begin{proof}
There are many ways to prove this but we give a proof that illustrates the usefulness of the Bennequin inequality. Notice that the surface $\Sigma$ has Euler characteristic $\chi(\Sigma_w)=n-k$ since we used $n$ disks (that is $n$, 0--handles) and $k$ bands (that is $k$, 1--handles) to build $\Sigma_w$. (Recall $k$ is the length of the braid word $w$.) Thus the genus of $\Sigma_w$ is $g(\Sigma_w)=\frac 12(k-n+1)$.

From Lemma~\ref{braidsl} we see that the closure of $w$ is a transverse link with self-linking number $sl(\overline{w})=k-n$ and hence we know by the Bennequin inequality from Theorem~\ref{bennequin} that any surface $\Sigma$ with $\partial  \Sigma=\overline{w}$ satisfies $2g(\Sigma)-1= -\chi(\Sigma)\geq sl(\overline{w})=k-n$. Thus we see that $\Sigma_w$ has minimal genus among all surfaces with boundary $\overline{w}$.
\end{proof}

We also notice another nice geometric fact about closures of positive braids.
\begin{thm}
Let $K$ be the closure of a non-split, positive braid. Then $K$ is a fibered link. 
\end{thm}
This theorem was originally due to Stallings in 1978 \cite{Stallings78} but can easily be seen using Gabai's criteria for knots being fibered \cite{Gabai86}.

\subsubsection{Quasi-positive generators}
The above is the standard notion of positivity for braids and we see that closures of positive braids have some nice properties, but one might ask why choose the above generating set to define positivity? Is there a more ``natural" one? that maybe says more about properties of ``positive" braids defined in term of these generators? 

Thinking of the braid group as a mapping class group and using the notation from the beginning of the section we can take the set $A_{qp}$ of half twists $h_\alpha$ about all arcs $\alpha$ in $D$ that intersect the $x_i$ only it its end points, see for example Figure~\ref{PQgen}. Clearly $A_{qp}$ generates the braid group $B(n)$. We let $QP(n)$ be the monoid generated by $A_{qp}$ and call any braid in $QP(n)$ a \dfn{quasi-positive braid}. Moreover any link that is the closure of a quasi-positive braid will be called a quasi-positive link. 

\begin{example} 
\begin{figure}[htb]
\begin{overpic}
{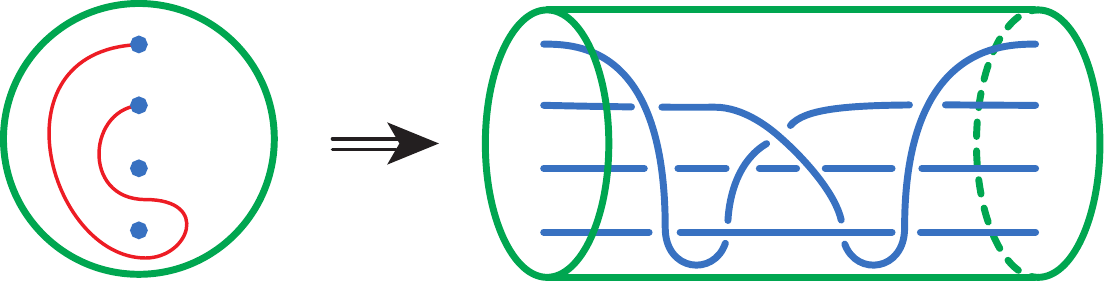}
\put(20,45){$\alpha$}
\end{overpic}
\caption{An arc $\alpha$ and the corresponding braid.}
\label{PQgen}
\end{figure}
  
Notice that $h_\alpha$ for $\alpha$ as in Figure~\ref{PQgen} can be written 
\begin{align*}
h_\alpha&= (\sigma_3\sigma_2\sigma_1\sigma_1\sigma_2^{-1})\sigma_3(\sigma_2\sigma_1^{-1}\sigma_1^{-1}\sigma_2^{-1}\sigma_3^{-1})\\
&= w\sigma_3 w^{-1}
\end{align*}
where $w=\sigma_3\sigma_2\sigma_1\sigma_1\sigma_2^{-1}$. So we see that $h_\alpha$ is a conjugate of the standard generator $\sigma_3$.
\end{example}
One may easily show that for any arc $\alpha$ as above the element $h_\alpha$ can be written as a conjugate of a generator $\sigma_i$. 
\begin{remark}
Notice that a braid is quasi-positive if the braid can be written as a word in the conjugates of the ``standard generators" $\sigma_i$. Or said another way $QP(n)$ is normally generated by the generators $\{\sigma_1,\ldots, \sigma_n\}$.
\end{remark}

We claim these generators are a more natural generating set since we are not preselecting a specific set of arcs as we did for the $\sigma_i$ but taking all arcs to define the generators. Of course one draw back to this generating set is that it is not finite. So $QP(n)$ is not finitely generated but it is finitely normally generated. 

Another argument that quasi-positivity is a more natural notion that positivity is that it has a surprising geometric interpretation. Specifically we will think of $S^3$ as a sphere of some radius in complex 2-space $\C^2$. Let $\Sigma$ be a complex curve in $\C^2$. If $\Sigma$ intersects $S^3$ transversely then we call $K=\Sigma\cap S^3$ a \dfn{transverse $\C$-link}, see Figure~\ref{complex}.

\begin{figure}[htb]
\begin{overpic}
{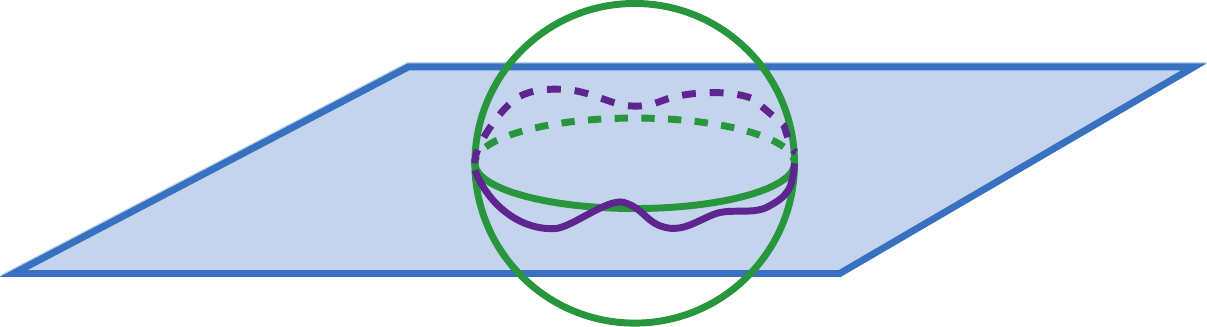}
\put(265,15){$\Sigma$}
\put(182,20){$K$}
\put(145,88){$S^3$}
\end{overpic}
\caption{The complex curve $\Sigma$ is represented by the horizontal blue rectangle and intersects $S^3$ transversely in the link $K$.}
\label{complex}
\end{figure}

We note that the class of transverse $\C$-links includes links of singularities, like the torus knots, but is a much bigger class of knots.  

We have the following amazing theorem.
\begin{thm}[Rudolph, 1983 \cite{Rudolph83}; Boileau and Orevkov, 2001 \cite{BoileauOrevkov01}]
The set of transverse $\C$-links in $S^3$ agrees with the set of quasi-positive links in $S^3$. 
\end{thm}
So quasi-positivity has a geometric meaning! The fact that quasi-positive links are transverse $\C$-links was shown by Rudolph and the fact that transverse $\C$-links are quasi-positive was shown by Boileau and Orevkov.

Orevkov has a method of using quasi-positive knots to study Hilbert's $16^\text{th}$ problem about the possible configurations of real algebraic planar curves, see \cite{Orevkov99}.

In particular as part of his program to study Hilbert's $16^\text{th}$ problem Orevkov asked \cite{Orevkov00} the following two questions. 
\begin{question}\label{orevkov1}
Given two quasi-positive braids representing the same fixed link, are they related by {\em positive} Markov moves and conjugation?
\end{question}

\begin{question}\label{orevkov2}
Given a quasi-positive link, is any braid representing the link with minimal braid index quasi-positive?
\end{question}

The answer to the first question is now known to be NO, but it is YES for a subset of quasi-positive links. Moreover a partial positive answer to the second question can also be given. All these results involve contact geometry and will be giving in Section~\ref{contact2studymonoid}.  We also note that Orevkov proved the following result.
\begin{thm}[Orevkov, 2000 \cite{Orevkov00}]\label{qpandstab}
A braid $B$ is quasi-positive if and only if any positive Markov stabilization of $B$ is quasi-positive. 
\end{thm}
Notice that a positive answer to Question~\ref{orevkov1} together with this theorem would say that questions about quasi-positive links could be answered purely in the terms of the quasi-positive monoids $QP(n)$. 

We now consider surfaces with boundary a quasi-positive link. Recall our construction of a Seifert surface $\Sigma_w$ for the closure of a braid $w$ above. We can try the same thing except the core of our attached bands will be given by the arcs $\alpha$ defining the generators of the quasi-positive monoid.
\begin{example}
Consider the 4-braid $w$ from Figure~\ref{PQgen}. To construct an immersed surface with boundary $\overline{w}$ we take 4 disks as shown in the middle of Figure~\ref{qusurface} and then attache a 1-handle with core given by $\alpha$. One easily sees that the resulting surface the desired boundary. 
\begin{figure}[htb]
\begin{overpic}
{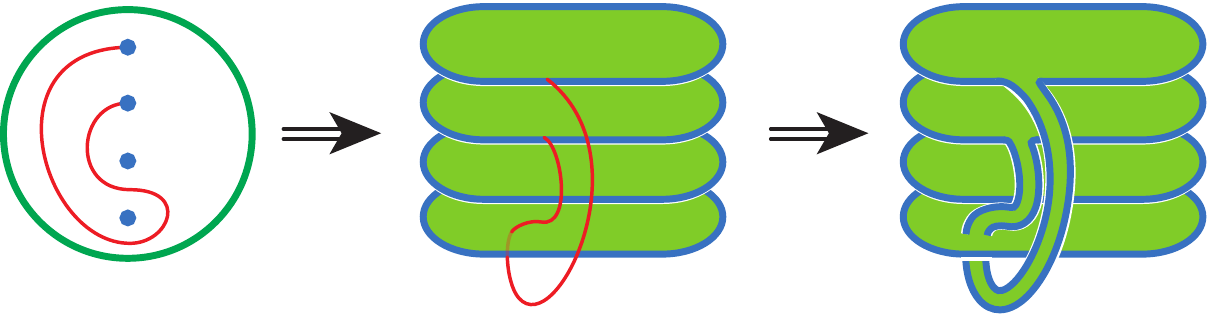}
\put(15,50){$\alpha$}
\put(165,5){$\alpha$}
\end{overpic}
\caption{The arc $\alpha$ that defines the diffeomorphism $h_\alpha$ shown on the left. The surface $\widetilde{\Sigma}_{\alpha}$ shown on the right.}
\label{qusurface}
\end{figure}
\end{example}
More generally given a word $w$ in the quasi positive generators we get an surface $\widetilde{\Sigma}_w$ for the closure of $w$. Notice that $w$ can also be represented as a word $w'$ in the standard generators of the braid group and from there we get a surface $\Sigma_{w'}$. Clearly the length of the word $w'$ is longer than (or equal to) the length of $w$ and so $\widetilde{\Sigma}_w$ will have genus less than or equal to the genus of $\Sigma_{w'}$. The main draw back to $\widetilde{\Sigma}_w$ is that it is {\em not embedded}. But notice that it only has ribbon singularities and no clasp singularities. (Recall that any immersed surface with embedded boundary in a 3-manifolds will have double points that are either of clasp or ribbon type. See Figure~\ref{rc}.) 
\begin{figure}[htb]
\begin{overpic}
{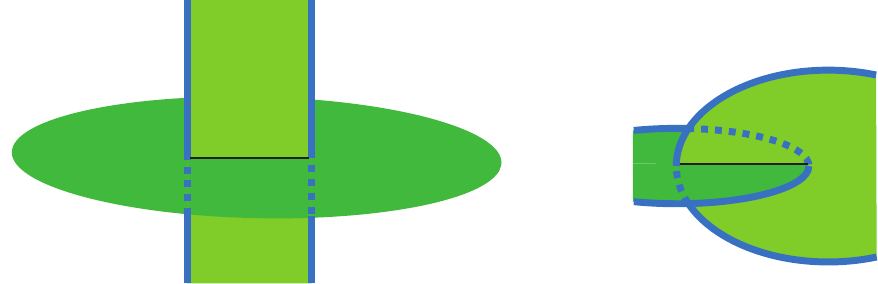}
\end{overpic}
\caption{A ribbon singularity shown on the left and a clasp singularity shown on the right.}
\label{rc}
\end{figure}
Notice that if we think of $\widetilde{\Sigma}_w$ as sitting in $S^3=\partial B^4$ then a piece of one of the sheets involved in the ribbon singularity can be pushed into the interior of the 4-ball so that $\widetilde{\Sigma}_w$ can be slightly perturbed, relative to the boundary, to be embedded in $B^4$. We now recall that following work of Rudolph \cite{Rudolph95} concerning slice knots, Lisca and Mati\'c and, independently,  Akbulut and Matveyev proved the ``slice Bennequin bound".
\begin{thm}[Lisca and Mati\'c, 1996 \cite{LiscaMatic96};  Akbulut and Matveyev, 1997 \cite{AkbulutMatveyev97}]
If $T$ is any knot transverse to the standard contact structure $\xi_{std}$ on $S^3$ then 
\[
sl(T)\leq -\chi(\Sigma),
\]
where $\Sigma$ is any smoothly embedded surface in $B^4$ with boundary $T$. 
\end{thm}
Thus we see that the self-linking number of a transverse knot $T$ also bounds the 4--ball genus (a.k.a. the slice genus), $g_4(T)$, of the knot. The proof of this theorem is much more difficult than the proof of Bennequin's original theorem and involves gauge theory. Mirroring the proof of Proposition~\ref{postive-min-genus} one can easily see the surface $\widetilde{\Sigma}_w$ constructed for a quasi-positive braid minimizes the 4--ball genus. 
\begin{prop}
If $w$ is a quasi-positive braid word then the surface $\widetilde\Sigma_w$ constructed above is a minimal genus surface in $B^4$ with boundary the closure of $w$. That is $g_4(\overline{w})=g(\widetilde\Sigma_w)$.
\end{prop}

\subsubsection{Strongly quasi-positive generators}
Our last class of ``positive braids" are the so called \dfn{strongly quasi-positive} braids. For these we take as a generating set for $B(n)$ the braids 
\[
\sigma_{ij}= (\sigma_i\ldots \sigma_{j-2})\sigma_{j-1}(\sigma_i\ldots \sigma_{j-2})^{-1},
\]
for $1\leq i<j<n$.  Or in terms of the mapping class group model of $B(n)$ the $\sigma_{ij}$ are just half twists along arcs that have non-positive $x$-coordinate, see Figure~\ref{SPQgen}. 
\begin{figure}[htb]
\begin{overpic}
{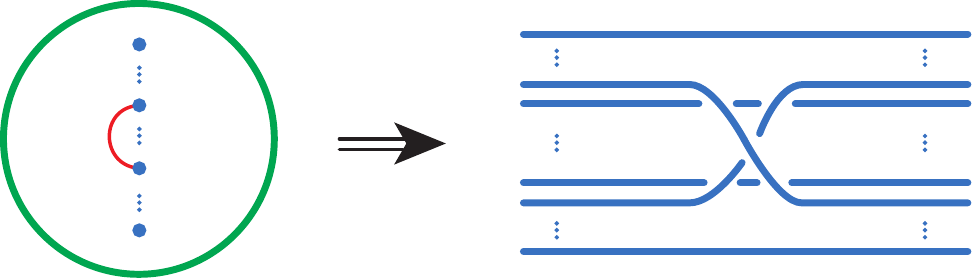}
\put(23,44){$\alpha$}
\put(45,29){$i$}
\put(142,19){$i$}
\put(45,49){$j$}
\put(142,55){$j$}
\end{overpic}
\caption{An arc $\alpha$ and the corresponding braid $\sigma_{ij}$.}
\label{SPQgen}
\end{figure}
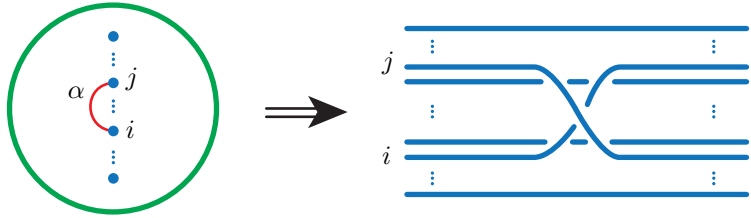
We call a braid \dfn{strongly quasi-positive} if it can be written as a positive word in the generators $\sigma_{ij}$. The monoid of strongly quasi-positive braids will be written $SQP(n)$. 

Notice that if $w\in SQP(n)$ then the surface $\widetilde{\Sigma}_w$ constructed in the last subsection is embedded in $S^3$ and so is a Seifert surface for the closure $\overline{w}$ of $w$. More to the point notice that the slice Bennequin inequality gives 
\[
\frac{\sl(\overline{w})+1}{2}\leq g_4(\overline{w})\leq g_3(\overline{w})
\]
and by construction $\frac{\sl(\overline{w})+1}{2}=g(\widetilde{\Sigma}_w)$. So for strongly quasi-positive knots the genus and 4-ball genus agree. In fact one can show that a quasi-positive knot is strongly quasi-positive if and only if its genus and 4-genus agree. One may easily use this to find examples of quasi-positive knots that are not strongly quasi-positive. Moreover since there are strongly quasi-positive knots that are not fibered it is clear that they do not have to be positive braids. 

\subsubsection{Summary}
We see that by choosing various generating sets for the braid group $B(n)$ we get a sequence of monoids with the following strict inclusions:
\[
P(n)\subset SQP(n)\subset QP(n) \subset B(n). 
\]
Moreover each of these monoids is associated to interesting geometric properties of the knots coming from the closures such braids. Namely:
\begin{enumerate}
\item The closure of a quasi-positive braid bounds a complex surface in the 4--ball and the natural ribbon surface built from the quasi-positive braid is a surface of minimal genus in the 4--ball with boundary the closed braid.
\item The closure of a strongly quasi-positive braid also bounds a complex surface in the 4-ball, its 4-ball genus and Seifert genus are equal, and the natural surface built form the strongly quasi-positive braid is a surface of minimal genus in $S^3$ with boundary the closed braid.
\item The closure of a positive braid is a fibered knot in $S^3$ with minimal genus Seifert surface coming form the braid presentation. 
\end{enumerate}
Later, in Section~\ref{contact2studymonoid}, we will see that one may use contact geometry to say something about these monoids. In particular, one can address Questions~\ref{orevkov1} and~\ref{orevkov2}.

\section{Monoids via orderings}\label{ordermonoid}

In this section we explore the use of (left-invariant) orderings and quasi-morphisms to construct monoids in groups. 

\subsection{Left-invariant orderings}
Recall that a \dfn{strict linear ordering} on a set $S$ is a relation $<$ such that 
\begin{enumerate}
\item the property $<$ is transitive and 
\item for each $a, b\in S$ we have exactly one of the following being true, $a<b, b<a$ or $b=a$.  
\end{enumerate}
We say such and ordering on a group $G$ is a \dfn{left-invariant ordering} if $g<h$ implies $kg<kh$ for any $g,h,k$ in $G$.

Now given a left-invariant ordering $<$ on a group $G$ we get a monoid 
\[
M_<=\{g\in G: e<g \, \, \text{ or }\, \,  g=e\},
\]
where $e$ is the identity element in $G$. Clearly if $g$ and $h$ are in $M_<$ then $g<gh$ by the left-invariant property and $1<g<gh$ by the transitivity property. Thus $M_<$ is closed under multiplication and clearly contains $e$. 

It turns out that the mapping class group of a surface with boundary (and possibly with marked points) has a left-invariant ordering. To see this, we discuss how to compare two arcs on a surface with the same end point. Let $c_1$ and $c_2$ be two arcs embedded in an oriented surface $S$ so that they share an end point $x\in \partial S$. The boundary of these arcs needs to be contained in the union of $\partial S$ and the marked points. We say that \dfn{$c_1$ is to the right of $c_2$ at $x$} if after isotoping $c_1$ and $c_2$, relative to their end points, so that they intersect minimally and orienting them away from $x$ we have that the oriented tangent to $c_1$ followed by the oriented tangent to $c_2$ gives an oriented basis for $T_xS$. See Figure~\ref{arcorder}. If $c_1$ is to the right of $c_2$ at $x$ then we denote this by $c_2<_x c_1$. 
\begin{figure}[htb]
\begin{overpic}
{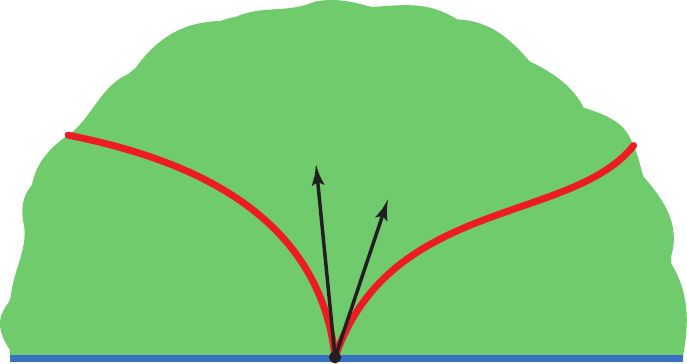}
\put(160,40){$c_1$}
\put(40,50){$c_2$}
\put(111,49){$c_1'(x)$}
\put(87,60){$c_2'(x)$}
\put(94,-7){$x$}
\end{overpic}
\caption{The arcs $c_1$ and $c_2$ sharing the endpoint $x$ in $\partial S$. If the surface has the ``counterclockwise" orientation then $c_2$ is less than $c_1$ at $x$, that is $c_1$ is to the right of $c_2$ at $x$.}
\label{arcorder}
\end{figure}
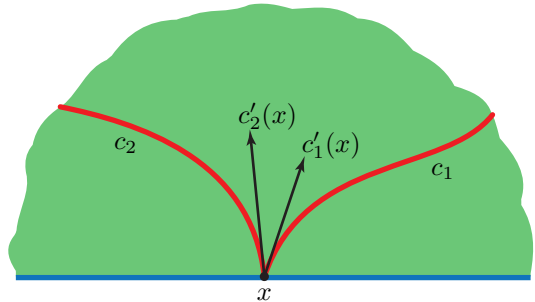

We now define a left invariant ordering on the braid group though of as the mapping class group $Mod(D_n,\partial D_n)$ that will be generalized to any (non-closed) surface later. Denoting the marked points in $D_n$ by $\{x_1,\ldots, x_n\}$ as in Section~\ref{braids}. Consider the arcs $c_1,\ldots c_n$ where $c_i$ has one end point on $x_i$, the other end point, which we denote $y_i$,  on $\partial D_n$ and has constant $y$-coordinate. See Figure~\ref{braidorder}. 
\begin{figure}[htb]
\begin{overpic}
{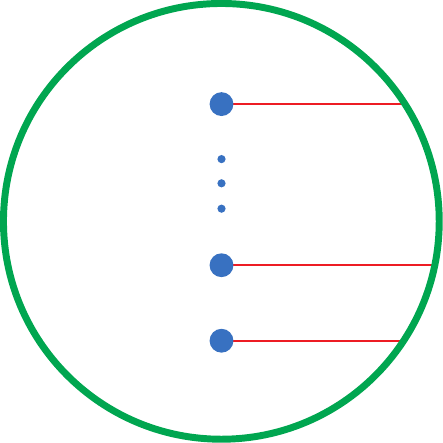}
\put(50,21){$x_1$}
\put(50,43){$x_2$}
\put(50,90){$x_n$}
\put(90,33){$c_1$}
\put(90,55){$c_2$}
\put(90,102){$c_n$}
\end{overpic}
\caption{The arcs $c_n$ used in the ordering of the braid group.}
\label{braidorder}
\end{figure}
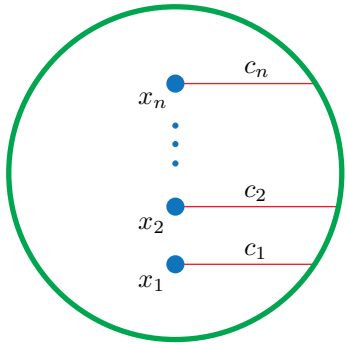
We say that the braid $b_1\in Mod(D_n,\partial D_n)$ is greater than the braid $b_2$ if for the first $i$ such that $b_1(c_i)\not=b_2(c_i)$ we have $b_2(c_i)<_{y_i}b_1(c_i)$. Notice that $b_1$ and $b_2$ are isotopic if and only if the image of all the arcs $c_i$ under $b_1$ and $b_2$ are isotopic (rel end points and the isotopies of the arcs are not allowed to cross the marked points). Thus we see that $b_1<b_2$, $b_2<b_1$ or $b_1=b_2$. So we clearly have a total order on $Mod(D_n,\partial D_n)$ that is obviously left-invariant. Notice that one would get a different order on $Mod(D_n,\partial D_n)$ if the arcs were ordered differently or if a different set of $n$ arcs were chosen. The first left-invariant orderings on the braid group was defined by Dehornoy in 1982 using an algebraic approach. In \cite{FennGreeneRolfsenRourkeWiest99} it was shown that Dehornoy's order is essentially equivalent to the one defined above. (It takes a little work to see this and the authors thank Dan Margalit for helping to confirm this observation.)

Generalizing the above example we construct orderings on the mapping class group $Mod(S,\partial S)$ for surfaces with boundary by comparing arcs in a basis for $S$ with their image under a mapping class element. More specifically consider a surface $S$ of genus $g$ with $k$ boundary components and marked points $\{x_1,\ldots x_n\}$. Once can choose $2g+k-1$ properly embedded disjoint arcs $\{\gamma_1,\ldots, \gamma_{2g+k-1}\}$ that cut $S$ into a disk with $n$ marked points. Now choose $n$ further disjoint arcs $\{\alpha_1,\ldots, \alpha_n\}$ that are disjoint from the $\gamma_i$ and such that $\partial \alpha_j$ consists of one point on the boundary of $S$ and $x_j$. We call the union of these arcs a basis for $S$ and denote it by $\mathcal{A}$. Now chose an ordering $\mathcal{O}$ of the boundary points of these arcs that are contained in $\partial S$: $\{y_1,\ldots, y_{4g+2k-2+n}\}$. See Figure~\ref{surfaceorder}. 
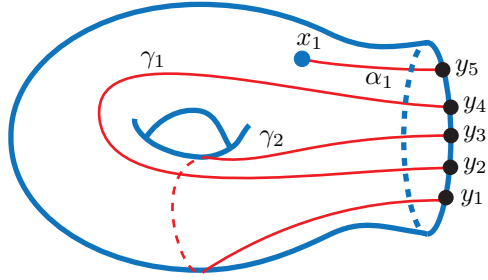
\begin{figure}[htb]
\begin{overpic}
{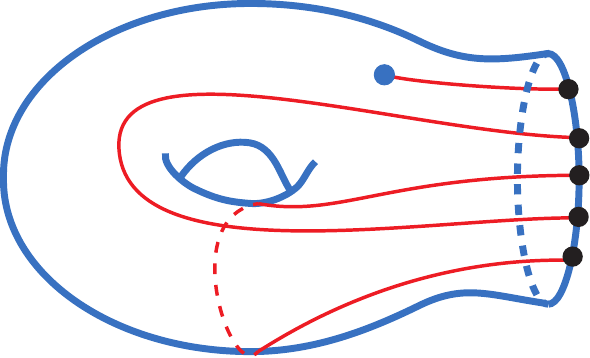}
\put(50,80){$\gamma_1$}
\put(95,50){$\gamma_2$}
\put(135,72){$\alpha_1$}
\put(109, 87){$x_1$}
\put(171,28){$y_1$}
\put(172,40){$y_2$}
\put(172,52){$y_3$}
\put(172,63){$y_4$}
\put(169,77){$y_5$}
\end{overpic}
\caption{The arcs used in defining the ordering of $Mod(S,\partial S)$.}
\label{surfaceorder}
\end{figure}

To define our left-invariant ordering on $Mod(S,\partial S)$ we need a couple of simple observations.
\begin{lem}\label{lem1}
Let $f$ and $g$ be two diffeomorphisms the surface $S$ that fix the boundary of $S$. If $f$ and $g$ act the same way on each arc in a basis $\mathcal{A}$ then $f$ is isotopic to $g$. (Here we say $f$ and $g$ act the same on an arc $c$ if $f(c)$ and $g(c)$ are isotopic relative to their end points.)
\end{lem}
\begin{lem}\label{lem2}
Let $c$ and $c'$ be two properly embedded arcs in $S$ that have the same end points. The arcs are isotopic if and only if neither $c<_x c'$ nor $c'<_x c$ at each end point $x$ of $c$. 
\end{lem}

We now define a left-invariant ordering $<_{\mathcal{A, O}}$ on $Mod(S,\partial S)$ as follows. If $f$ and $g$ are two diffeomorphisms of $S$ the was say $f<_{\mathcal{A, O}}g$ if there is some $i$ in $\{1,\ldots, 4g+2k-2+n\}$ such that the arc $\beta$ in $\mathcal{A}$ with end point $y_i$ satisfies $f(\beta)<_{y_i} f(\beta)$ and all the arcs $\mathcal{A}$ with endpoints $y_j$ for $j<i$ are fixed (up to isotopy) by $f\circ g^{-1}$. 

More informally we say that $f<_{\mathcal{A, O}}g$  if running through the end points in the order $\mathcal{O}$ the first time $f$ and $g$ act differently on an arc in the basis then $g$ moves the arc to the right of $f$ at that end point. Notice that according to Lemmas~\ref{lem2} if neither $f<_{\mathcal{A, O}}g$ nor $g<_{\mathcal{A, O}}f$ then $f$ and $g$ act the same on all the arcs in $\mathcal{A}$ and then Lemma~\ref{lem1} implies that $f$ and $g$ are isotopic. So we see that $<_{\mathcal{A, O}}$  is a strict linear ordering. One may easily check that the ordering is left-invariant since applying a diffeomorphism to a pair of arcs will not change whether or not one is to the right or left of the other. Thus we have a left-invariant order on $Mod(S,\partial S)$ and the corresponding monoid
\[
M_{<_\mathcal{A,O}}
\]
that depends on the basis $\mathcal{A}$ and ordering $\mathcal{O}$ of the endpoints of the arcs. 

It is clear that the intersection of monoids in a group is also a monoid. So to remove the dependence of the basis and ordering when defining $M_{<_\mathcal{A,O}}$ we can define the \dfn{right-veering monod} as follows
\[
Veer^+(S)= \bigcap_{\mathcal{A},\mathcal{O}} M_{<_\mathcal{A,O}},
\]
where the intersection is taken over all bases $\mathcal{A}$ for $S$ and all orderings $\mathcal{O}$ of the endpoints. If $\phi:S\to S$ is a diffeomorphism in $Veer^+(S)$ then we call it \dfn{right-veering} (though it is probably more accurate to call it non-left-veering).  Once can easily verify that $\phi$ is right-veering if and only if for every embedded arc $\gamma$ the image $\phi(\gamma)$ is to the right of $\gamma$ at each end point (or isotopic to $\gamma$). Since the work of Honda, Kazez and Mati\'c \cite{HondaKazezMatic07} the notion of right-veering has become a central notion in contact geometry due in large part to the following fundamental theorem.
\begin{thm}[Honda, Kazez and Mati\'c, 2007 \cite{HondaKazezMatic07}]
If $\xi$ is a tight contact structure on a closed 3--manifold, then any open book supporting $\xi$ has right-veering monodromy.
\end{thm}
So once again we see that a natural monoid has connections to interesting geometric properties. It is easy to see that there are many open books for overtwisted contact structures that are also right-veering so this monoid does not completely characterize right-veering, but it does provide quite a bit of insight into the tight vs.\ overtwisted dichotomy. 

\subsection{Quasi-morphisms}
Notice that if we have a group $G$, another group $H$ with a left-invariant ordering $<$, and a homomoprhism $f:G\to H$, then we can construct monoids in $G$ as we did for $H$. That is we can set
\[
M_{f,<}=\{g\in G: g=e_G \text{ or } e_H<f(g)\},
\]
where $e_G$ is the identity element in $G$ and analogously for $e_H$. But one can get by with much less. Suppose that we have a quasi-morphism from a group $G$ to the real line $\R$. This is simply a map of sets 
\[
q:G\to \R
\]
such that there is some constant $C$ that satisfies 
\[
|q(g_1g_1)-q(g_1)-q(g_2)|<C
\]
for all $g_1, g_2\in G$. (Notice if $C=0$ then $q$ is a homomorphism.) Given this and a number $r$ can now consider the sets
\[
M_{r,q}=\{g\in G: g=e \text{ or } q(g)\geq r\}.
\]
It should be clear that if $r\geq C$ then $M_{r,q}$ is a monoid. 

We now consider such a quasi-morphism on the mapping class group. 
Given a diffeomorphism of a surface $\phi:S\to S$ that is the identity on the boundary one can define the fractional Dehn twist coefficient (or FDTC for short) of $\phi$ relative to a boundary component $C$ of $S$. We denote this by $c(\phi, C)$. This was originally studied in Gabai and Ortel's work \cite{GabaiOertel89} on essential laminations and then Roberts \cite{Roberts01, Roberts01b} when studying taut foliations. Its most modern incarnation occurred in work of Honda, Kazez and Mati\'c \cite{HondaKazezMatic07, HondaKazezMatic08} in relation to contact geometry. 

There are several definitions of the FDTC. We give a simple topological definition and then state a few properties that are useful for computations and our discussion of monoids. Recall that according to the Nielsen-Thurston classification of surface diffeomorphisms a diffeomorphism $\phi:S\to S$ is freely isotopic to a diffeomorphism $h:S\to S$ that is (1) periodic, (2) pseudo-Anosov, or (3) reducible. (Here freely isotopic means the isotopy can move the boundary.) Recall $h$ is periodic if there is an $n$ such that $h^n$ is freely isotopic to the identity on $S$ and it is reducible if there is a non-empty collection of curves on $S$ that is preserved by $h$. The diffeomorphism is pseudo-Anosov if there is a pair of measured geodesic laminations $\lambda_s$ and $\lambda_u$ that are preserved and dilated in certain ways by $h$, see \cite{HondaKazezMatic07} for more details. (Notice that a periodic diffeomorphism is technically reducible too, so when we call $h$ reducible we will mean that it is not periodic and preserves a multi-curve.) In the case that $h$ is pseudo-Anosov we define the {fractional Dehn twist coefficient} for the boundary component $C$ of $S$ as follows. Notice that the diffeomorphism $h$ induces a flow on the mapping torus $T_\phi$ of $\phi$. When restricted to the torus $C\times S^1$ in $\partial T_\phi$ this flow will have rational slope and hence its flow lines will contain parallel closed curves. Let one be $\gamma$. Recall that we have a natural longitude $\lambda=C\times\{pt\}$ and meridian $\mu=\{pt\}\times S^1$ for $C\times S^1$ and hence $\gamma$ is homologous to $p\lambda+q\mu$. The \dfn{fractional Dehn twist coefficient}  of $\phi$ along $C$ is
\[
c(\phi, C)=\frac pq.
\]
There is a similar definition for reducible and periodic diffeomorphisms. For more details on the definition see \cite{HondaKazezMatic07}, for now, we will simply focus on important properties of the FDTC. Specifically it is know that the FDTC give a quasi-morphism from the mapping class group to the rational numbers. That is if you fix a boundary component $C$ of $S$ then 
\[
c(\cdot, C): Mod(S,\partial S)\to \Q
\] 
satisfies 
\[
|c(\phi\circ \psi, C)-c(\phi,C)-c(\psi,C)|\leq 1.
\]
This seems to be a well known folk result, but a nice proof of it can be found in \cite{ItoKawamuro12}.

Now for any $r\in \R$ and surface $S$ with boundary define 
\begin{align*}
FDTC_r(S)=\{\phi\in Mod(S,\partial S):& \, \phi=id_S \text{ or } \\ & c(\phi, C)\geq r \text{ for all components $C$ of $\partial S$}\}.
\end{align*}
The quasi-morphism condition implies that for $r\geq 1$, $FDTC_r(S)$ is a monoid!

In \cite{HondaKazezMatic07}, it was shown that in both the periodic and pseudo-Anosov case a diffeomorphism $\phi$ was right veering if and only if all of its fractional Dehn twist coefficients were non-negative. This easily implies the same for the reducible case too. Thus one may easily conclude that
\[
FDTC_0(C)=Veer^+(S),
\]
so it is a monoid too. 

\section{Monoids via contact geometry}\label{mvcg}

Our third method for constructing monoids in the mapping class group is a bit unexpected. We have already seen that there are relations between contact geometry and monoids but one can actually construct monoids via contact geometry. 

Recall from Section~\ref{girouxcor} that given an element $\phi$ in the mapping class group $Mod(S,\partial S)$ one can construct a 3--manifold  $M_\phi$ and a contact structure $\xi_\phi$ on it. Now given a property $\mathcal{P}$ of a contact structure one can define a subset of $Mod(S,\partial S)$
\[
M_\mathcal{P}(S)=\{\phi\in Mod (S,\partial S): \xi_\phi \text{ has property } \mathcal{P} \}.
\]
A main question we have is when is $M_\mathcal{P}(S)$ a monoid. Before answering this we note a few examples in Table~\ref{contactproperties1}. 
\begin{table}[htdp]
\begin{center}
\begin{tabular}{c c}
Property $\mathcal{P}$ & Name for the subset $M_\mathcal{P}(S)$\\
\hline
Stein fillability & $Stein(S)$\\
Strong fillability & $Strong(S)$\\
Weak fillability & $Weak(S)$\\
Tightness & $Tight(S)$\\
Non-zero Ozsv\'ath-Szab\'o contact invariant & $OzSz(S)$\\
Universal Tightness & $UT(S)$\\
Tight but virtually overtwisted & $VOT(S)$\\
Overtwisted & OT(S)\\
\hline
\end{tabular}
\end{center}
\caption{Contact geometric subsets of the mapping class group. 
}\label{contactproperties1}
\end{table}%
Here we recall that to a contact structure $\xi$ on a 3--manifold $M$ there is an element $c(\xi)$ in the Heegaard-Floer groups $\widehat{HF}(-M)$ that is zero if the contact structure is overtwisted, \cite{OzsvathSzabo05a}. So $OzSz(S)$ corresponds to monodromies of open books with $c(\xi)\not=0$. 

Recall from Section~\ref{dt} we have the monoid $Dehn^+(S)$ consisting of mapping class elements that can be written as the composition of right handed Dehn twists and from Section~\ref{ordermonoid} we have the monoid $Veer^+(S)$ of right veering diffeomorphisms. These two monoids and the sets above are related according to the following diagram (where the surface $S$ has been suppressed from the notation for the sake of space)
\[
\begin{tikzcd}[row sep=tiny, column sep=small]
					&					&			&OzSz \arrow[hook]{dr}{(5)}	\\
Dehn^+ \arrow[hook]{r}{(1)}& Stein\arrow[hook]{r}{(2)}& Strong \arrow[hook]{ur}{(3)}\arrow[hook]{dr}{(4)} & & Tight \arrow[hook]{r}{(7)}&Veer^+.	\\
					&					&			&Weak \arrow[hook]{ur}{(6)}
\end{tikzcd}
\]
All the arrows represent inclusions. Inclusion (1) follows from \cite{Eliashberg90a, Giroux02}, inclusions (2) and (4) are obvious, (3) follows from \cite{OzsvathSzabo04a}, (5) form \cite{OzsvathSzabo05a}, (6) from \cite{Eliashberg90a, Gromov85} and (7) from \cite{HondaKazezMatic07}. It is also known that all the inclusions are strict. The strictness of (7) comes from \cite{HondaKazezMatic07}, (5) and (6) follow from \cite{Ghiggini06} and \cite{EtnyreHonda02b}, respectively, while (3) and (4) follow from \cite{Ghiggini06b}  and \cite{Eliashberg96}, respectively and (2) was shown to be a strict inclusion in \cite{Ghiggini05}. Lastly the strictness of (1) follows from \cite{BakerEtnyreVanHorn-Morris12, Wand??}. We also note that 
\[
Tight(S)=UT(S)\cup VOT(S).
\]
We now return to our question as to which $M_\mathcal{P}(S)$ are monoids. 
\begin{thm}[Baker, Etnyre, and Van Horn-Morris 2012, \cite{BakerEtnyreVanHorn-Morris12}; and Baldwin 2012, \cite{Baldwin12}]\label{monoidsfromcont}
Let $\mathcal{P}$ be a property of a contact structure. Then $M_\mathcal{P}(S)$ is a monoid if and only if $\mathcal{P}$ is preserved under (possibly internal) connected sums and Legendrian surgery (and $\xi_{id_S}$ satisfies the property). 
\end{thm}
It is well known that the first three and the fifth property in Table~\ref{contactproperties1} are preserved under connected sum and Legendrian surgery and in \cite{Wand14?} Wand showed the same for the fourth property. Thus by Theorem~\ref{monoidsfromcont} we see that the first 5 properties in the table define monoids. (That $OzSz(S)$ is a monoid was previous shown by Baldwin in \cite{Baldwin08}.) It is also not hard to see that the last three properties do not give monoids.  Specifically it is well known that overtwistedness is not preserved by Legendrian surgery and Gompf in \cite{Gompf98} gave examples of Legendrian surgeries on universally tight contact structures that resulted in virtually overtwisted contact structures. The following example shows that $VOT(S)$ is not a monoid. 

\begin{example}\label{example}
Consider the planar surface $S$ in Figure~\ref{FDTCex} and the curves $\gamma_i$ shown there too. Let $\phi=\tau_{\gamma_1}^2\tau_{\gamma_2}^2\tau_{\gamma_3}\tau_{\gamma_4}\tau_{\gamma_5}^{-1}$. One may easily check that this open books supports the virtually overtwisted contact structure on the lens space $L(4,1)$ (that is the contact structure coming from Legendrian surgery on the $tb=-3, r=0$ unknot in $S^2$, and the given monodromy comes form the ``obvious monodromy" by a lantern relation). Now let $\phi'$ and $\phi''$ be the monodromies obtained by rotating Figure\ref{FDTCex} by $\frac{2\pi}{3}$ and $\frac{4\pi}{3}$, respective. Each of these monodromies gives a virtually overtwisted contact structure on $L(4,1)$. The composition of all of the monodromies gives (after applying a lantern relation) $\phi\circ\phi'\circ\phi''= \tau_{\gamma_1}^4\tau_{\gamma_2}^4\tau_{\gamma_3}^4\tau_{\gamma_4}^2$.
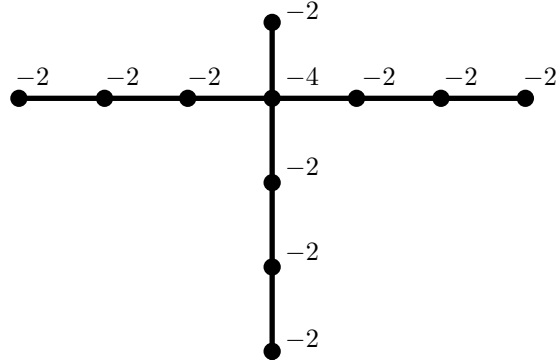
\begin{figure}[htb]
\begin{overpic}
{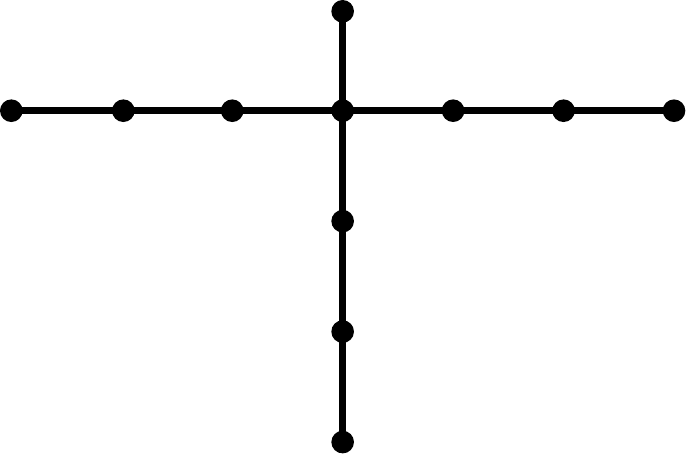}
\put(104,70){$-2$}
\put(104,38){$-2$}
\put(104,5){$-2$}
\put(104,129){$-2$}
\put(104,104){$-4$}
\put(133,104){$-2$}
\put(164,104){$-2$}
\put(194,104){$-2$}
\put(2,104){$-2$}
\put(36,104){$-2$}
\put(67,104){$-2$}
\end{overpic}
\caption{Plumbing diagram of the Milnor fillable contact structure.} 
\label{fig:plumbing}
\end{figure}
 
In \cite{EtguOzbagci06}, Etg\"u and Ozbagci show how to produce a planar open book decomposition with positive monodromy on any manifold described by a plumbing along a tree with no bad vertices. Reversing their construction, one can see that the open book with monodromy $ \tau_{\gamma_1}^4\tau_{\gamma_2}^4\tau_{\gamma_3}^4\tau_{\gamma_4}^2$ corresponds to an open book coming from their construction on the plumbing given in Figure \ref{fig:plumbing}. This is a negative definite plumbing along a tree with no bad vertices and by \cite{BhupalOzbagci11, EtguOzbagci06} one can see that the 
open books are horizontal and support the Milnor fillable contact structure on the corresponding Seifert-fibered space.

Lek{\i}l{\i} and Ozbagci \cite{LekiliOzbagci10} show that Milnor fillable contact structures are universally tight. Alternatively, since the above open book is horizontal the contact structure is transverse to the fibers of the Seifert fibration and one can also conclude the contact structure is universally tight using a result of Massot from \cite{Massot08}.  Combining these observations shows that the composition of the three virtually overtwisted monodromies $\phi\circ\phi'\circ\phi''$ yields a universally tight contact structure and so for the four-holed sphere, $VOT(S)$ is not a monoid, and one can extend this to most other surfaces by adding one handles and extending by the identity. \end{example}

There are several proof of Theorem~\ref{monoidsfromcont}, but the basic idea for the ``hard" direction is that one can construct the contact manifold supported by $(S,\phi\circ\psi)$ from the disjoint union of the contact manifolds supported by $(S,\phi)$ and $(S,\psi)$ by a sequence of the operations in the theorem.

\section{Questions about monoids}
One might now be interested in the structure of the monoids discussed above. Here we discuss a few obvious questions.
\begin{question}
Are the monoids above ``easily" presented? Are any finitely presented or finitely generated?
\end{question}
It is known \cite{Baldwin07, HondaKazezMatic09a} that when $S$ has genus 1 and 1 boundary component, then 
\[
OzSz(S)=Tight(S)=Veer^+(S).
\]
It is also known that $Veer^+(S)-Dehn^+(S)$ is non-empty. So one is naturally left to ask the following question.
\begin{question}
What is the relation between the monoids $Dehn^+(S)$, $Stein^+(S)$, $Strong(S)$, $Weak(S)$ and $OzSz(S)$ when $S$ is genus one with one boundary component? What are their generators?
\end{question}

In \cite{Baldwin07} it was shown that for a surface $S$ with of genus one with one boundary component $Tight(S)$ is {\em normally} generated by 
\[
\tau_a, \tau_b, (\tau_a\tau_b)^3\tau_b^{-n}, \text{ for } n\in \Z,
\]
where $a$ and $b$ are simple closed curves in $S$ that intersect once.
From this one can easily show the following.

\begin{thm}
For any surface $S$ with boundary $Tight(S)$ is not finitely generated.
\end{thm}
\begin{proof} Using properties of non-left veering maps, one can show that right handed Dehn twists along homologically essential, simple closed curves are initial elements in the right-veering monoid $Veer^+(S)$, that is there are no non-trivial elements less than them. To see this, we show that any factorization of $D_\gamma$ into non-left veering maps consists precisely of $D_\gamma$ (and the identity). Specifically, if $\alpha$ is a proper arc which is fixed by $D_\gamma$ (that is, it is disjoint from $\gamma$), then any right-veering factor of $D_\gamma$ also must fix $\alpha$. Thus all factors of $D_\gamma$ are supported in an annulus neighborhood of $\gamma$. 

In general, this shows that \emph{any} submonoid of $Veer^+(S)$ which contains $Dehn^+(S)$ must include Dehn twists about all simple closed curves in its generating set and so cannot be finitely generated.
\end{proof}
\begin{question} 
Do any of the contact monoids above have a finite, normally generating set? For example, $Dehn^+$ is normally generated by a single Dehn twist. \end{question}

Notice the above presentation shows that $OzSz(S)$, for $S$ a genus one, one boundary component surface, can be generated by the elements in $Dehn^+(S)$ together with $(\tau_a\tau_b)^3\tau_b^{-n}, \text{ for } n\in \Z$. This brings up the following natural question. 
\begin{question}
If $\mathcal{M}_i$, $i=1,2$, are two of the monoids above with $\mathcal{M}_1\subset \mathcal{M}_2$ then what elements from $\mathcal{M}_2$ must be added to $\mathcal{M}_1$ to generate all of $\mathcal{M}_2$? (We can think of this as asking what a generating set for a monoid being generated over a submonoid.)
\end{question}

Continuing with our natural questions we have the following. 
\begin{question}
Can you characterize when any $\phi\in Mod(S,\partial S)$ is in one of the above monoids? In particular given $\phi$ are there conditions on $\phi$ that will imply that $\xi_\phi$ is tight? Does it help if you restrict $S$ to be planar? or of small genus? or with a small number of boundary components?
\end{question}
Or a possibly simpler question is the following. 
\begin{question}
If $\phi$ is in one of the above monoids is there a condition that would force it into a sub-monoid?
\end{question}
In \cite{Wendl10} Wendl showed that for a planar $S$ we have 
\[
Dehn^+(S)=Stein(S)=Strong(S)
\]
and then in \cite{NiederkrugerWendl11} this was extended to include
\[
Strong(S)=Weak(S). 
\]

\begin{question}
Are there other interesting monoids in $Mod(S,\partial S)$ that correspond to something in the contact (or symplectic, or complex, or Riemannian) world?
\end{question}

\begin{question}
Can you use information about the fractional Dehn twist coefficients of a diffeomorphism to help with any of the above questions?
\end{question}
For example in \cite{HondaKazezMatic08} Honda, Kazez and Mati\'c showed that if $S$ has only one boundary component then $FDTC_1(S)\subset Weak(S)$; however it is known that there are $\phi\in Weak(S)-FDTC_1(S)$ and that if $S$ has more than one boundary component then $FDTC_1(S)$ does not have to be contained in $Weak(S)$, in fact Examples~\ref{fdtcexex} shows that there are $\phi\in FDTC_1(S)$ that are not even in $Tight(S)$. 
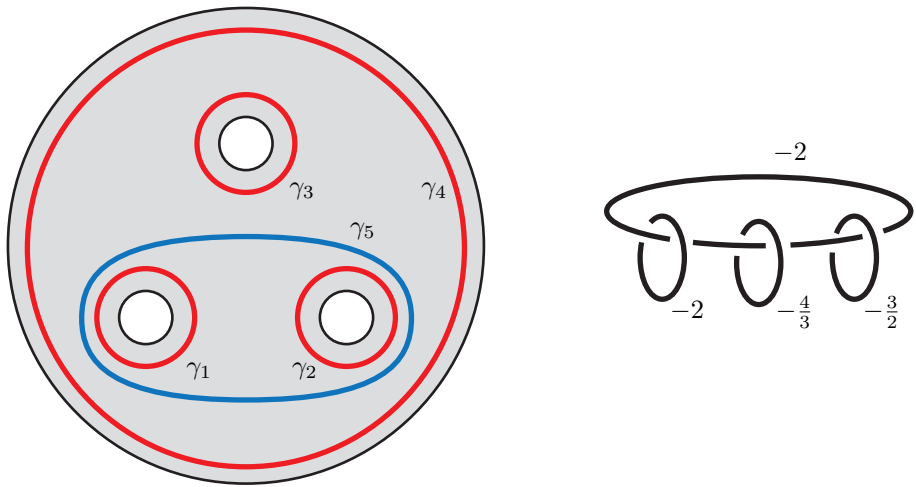
\begin{figure}[htb]
\begin{overpic}
{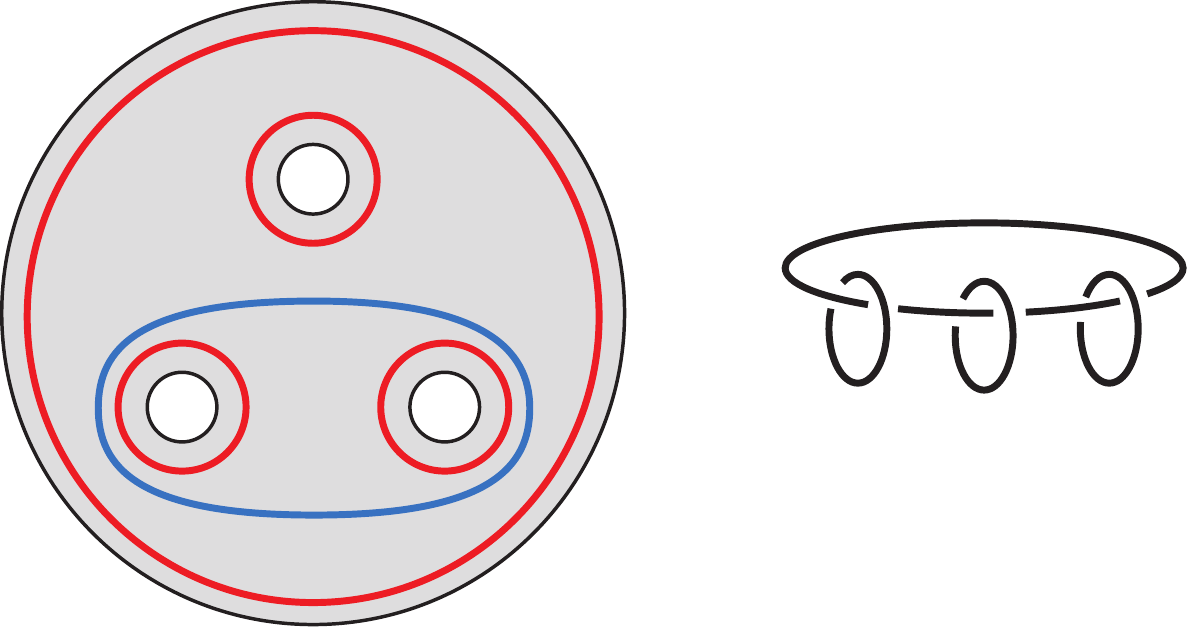}
\put(68,42){$\gamma_1$}
\put(108,42){$\gamma_2$}
\put(107,110){$\gamma_3$}
\put(157,110){$\gamma_4$}
\put(130,95){$\gamma_5$}
\put(251,63){$-2$}
\put(291,63){$-\frac 43$}
\put(324,63){$-\frac 32$}
\put(290,123){$-2$}
\end{overpic}
\caption{Left, planar surface with four boundary components and curves used to describe the monodromy in Example~\ref{fdtcexex}. Right the manifold constructed in the example.}
\label{FDTCex}
\end{figure}
\begin{ex}\label{fdtcexex}
Let $S$ be the surface in Figure~\ref{FDTCex} and $\phi=\tau_{\gamma_1}^2\tau_{\gamma_2}\tau_{\gamma_3}^3\tau_{\gamma_4}^2\tau_{\gamma_5}^{-2}$, where the $\gamma_i$ are also shown in the figure. One may easily check that this manifold $M$ associated to the open book $(S, \phi)$ is the Seifert fibered space $M(-2;1/2,2/3, 3/4)$. Let $\xi$ be the supported contact structure. From \cite{TosunPre} it is know that any tight contact structure on $M$ is Stein fillable, but the Oszv\'ath-Szab\'o contact invariant of $\xi$ is zero (this can be seen using \cite{Baldwin13} by capping off the boundary component of $S$ parallel to $\gamma_1$ and noting the resulting open book supports an overtwisted contact structure). But since Stein fillable contact structures must have non-vanishing contact invariant we see that $\xi$ is overtwisted. One may easily check that the FDTCs of $\phi$ at the boundary component parallel to $\gamma_2$ is 1, while all other FDTCs are greater than 1. 
\end{ex}

Moreover Kawamuro and Ito \cite{ItoKawamuro14} have shown that if $S$ is planar with any number of boundary components, then $FDTC_r(S)\subset Tight(S)$ for any $r>1$, but the example mentioned above shows that $FTDC_1(S)\not\subset Tight(S)$.

From the above results one might hope that if a diffeomorphism is in some monoid then mild hypothesis on the FTDC might promote it to a smaller monoid. 

Colin and Honda gave similarly strong results on the tightness of a contact structure by looking at open books with connected binding.

\begin{thm}[Colin-Honda 2013 {\cite[Theorem 4.2]{ColinHonda13}}] Let $S$ be a surface with connected boundary and $h$ a mapping class element. If $h$ is periodic, then $(M, \xi_{(S, h)})$ is tight if and only if $h$ is right-veering. Moreover, the tight contact structures are Stein fillable. \end{thm}

There are similar results that hold when the monodromy is pseudo-Anosov. Colin and Honda use the growth rates of the 
generators for the contact homology complex (as one increases the action) to show the following. 

\begin{thm}[Colin-Honda, 2013 {\cite[Theorem 2.3 and Corollary 2.7]{ColinHonda13}}] Let $S$ be a surface with connected boundary and $h$ a mapping class element. If $h$ is (isotopic to) a pseudo-Anosov diffeomorphism with fractional Dehn twist coefficient $k/n$, then
\begin{itemize}
\item If $k = 2$, then $(M, \xi_{(S, h)})$ is tight. 
\item If $k \geq 3$, then $(M, \xi_{(S, h)})$ is universally tight (and the universal cover of $M$ is $\R^3$). 
\end{itemize}
\end{thm}

We also have the following question about fractional Dehn twist coefficients. 
\begin{question}
Are the sets $FDTC_r(S)$ monoids for any $r\in (-\infty, 0)\cup (0,1)$?
\end{question}

\section{Monoids in the braid group via contact geometry}\label{bgmonoids}

As we saw in Section~\ref{ht}, there are some natural monoids in the braid group which come from generating sets, and each has a connection to the smooth topology and contact geometry of knots and links. Much like for surfaces, though, there are other monoids in the braid group coming from contact geometry and various knot homologies. In the first subsection we discuss monoids in the braid group analogous to those constructed in Section~\ref{mvcg} using contact geometry while in the following subsection we discuss using co-product operations in various homology theories to construct monoids. We thank Liam Watson for help with the foundational work of many of the ideas presented in this section. 

\subsection{Monoids in the braid group and transverse knots}
As in the Section~\ref{mvcg} if $\mathcal{P}$ is a property of a transverse knot then we can consider the subset of the braid group $B(n)$
\[
M_\mathcal{P}(n)=\{ w\in B(n): {\overline{w}} \text{ has the property } \mathcal{P}\}.
\]
We are interested in when these sets are monoids.  Some examples of such properties are given in Table~\ref{lab}.
\begin{table}[htdp]
\begin{center}
\begin{tabular}{c c}
Property $\mathcal{P}$ & Name for the subset $M_\mathcal{P}(n)$\\
\hline
Equality in slice Bennequin bound& $\Chi(n)$\\
Equality in $s$-invariant bound& $S(n)$\\
Equality in $t$-invariant bound&$T(n)$\\
Non-zero $\psi$ invariant&$\Psi(n)$\\
$k$--fold cyclic cover is Stein fillable & $Stein(n,k)$\\
$k$--fold cyclic cover is strong fillable & $Strong(n,k)$\\
$k$--fold cyclic cover is weak fillable & $Weak(n,k)$\\
$k$--fold cyclic cover is tight & $Tight(n,k)$\\
$k$--fold cyclic cover has non-zero & \\
Ozsv\'ath-Szab\'o contact invariant  & $OzSz(n,k)$\\
\hline
\end{tabular}
\end{center}
\caption{Transverse knot defined subsets of the braid group. 
}\label{lab}
\end{table}%
The precise concepts used to define the subsets above will be discussed in below. For now we determine that these are all monoids (for $T(n)$ this is true modulo a well-known and believed conjecture). 
\begin{thm}\label{braidmonoid}
\label{brmon} Let $\mathcal{P}$ be a property of braids which is preserved under transverse isotopy (of the closure), disjoint union (that is stacking the braids) and appending quasi-positive half twists. Then the subset of the braid group $B(n)$ consisting of all braids which satisfy $\mathcal{P}$ is a monoid. 
\end{thm}
Here we use \emph{stacking braids} to mean taking two $k$--braids $w_1$ and $w_2$ and diagrammatically putting one on top of the other to form a $2k$--braid.

\begin{proof} 
If we stack two $k$--braids $w_1$ and $w_2$ on top of each other to form a $2k$--braid, we can append $k$ quasi-positive half twists to form a braid which is positively Markov equivalent to the composition $w_1w_2$. This is illustrated in Figure~\ref{bmonoid} for the $k=3$ case. In the figure we start with $w_2$ and then turn it upside down (the closure of this braid gives the same transverse knot). In the next diagram we then conjugate $w_2$ by an a half twist and stack $w_1$ on top. We now isotope $w_1$ to the lower 3 strands. Finally we add quasi-positive bands to cancel the quasi-negative bands. This results in a braid that is positive Markov equivalent to the braid $w_1w_2$. We note the idea behind this proof comes from Figure~5 of \cite{Baldwin10} or similarly Figure~8 of \cite{BakerEtnyreVanHorn-Morris12}.
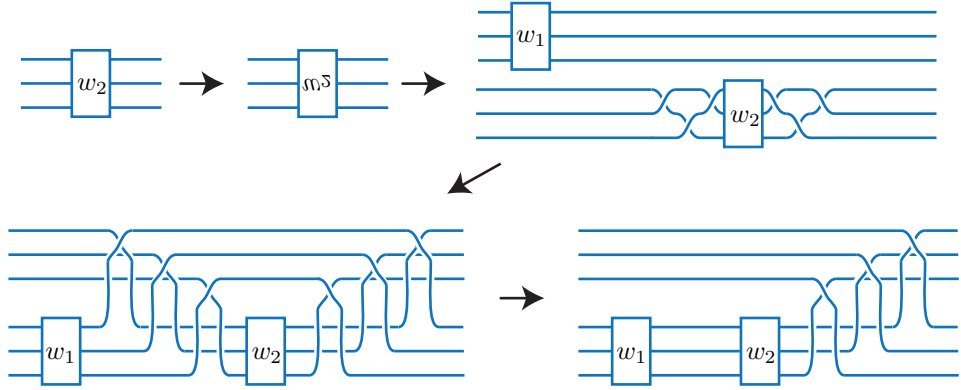
\begin{figure}[htb]
\begin{overpic}
{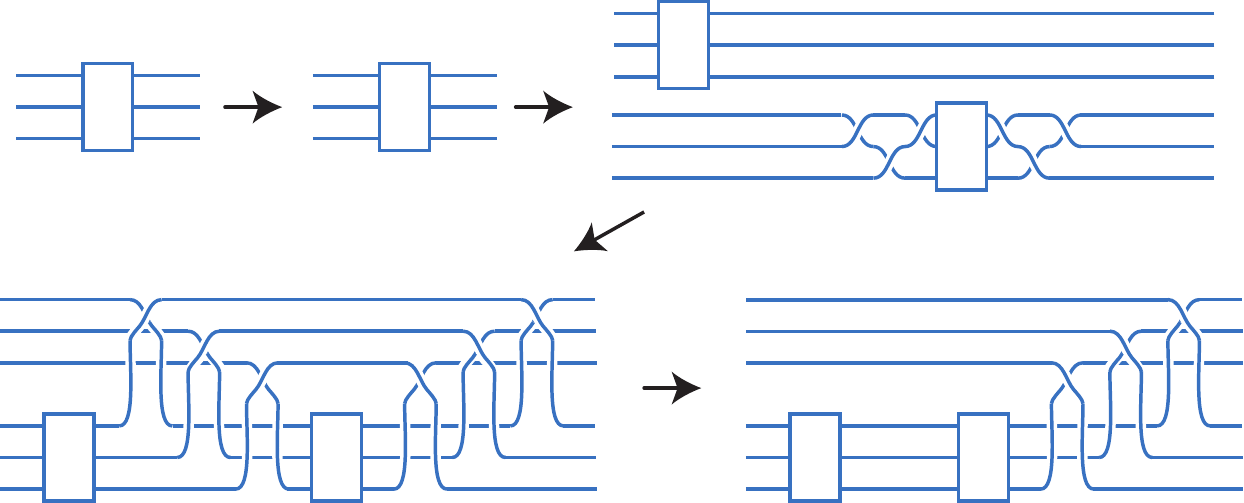}
\put(26,112){$w_2$}
\put(111,116){\rotatebox{180}{\reflectbox{$w_2$}}}
\put(273,100){$w_2$}
\put(192, 130){$w_1$}
\put(14,10){$w_1$}
\put(92,10){$w_2$}
\put(230,10){$w_1$}
\put(279,10){$w_2$}
\end{overpic}
\caption{Constructing $w_1w_2$ from $w_1$ and $w_2$ using stacking and appending quasi-positive half twists.}
\label{bmonoid}
\end{figure}

\end{proof}

\subsubsection{Bennequin type inequalities} There is a braid group analogue of the question of whether $Dehn^+$ and $Stein$ are the same monoid. While any quasi-positive presentation of a braid produces a slice surface which realizes the slice-Bennequin bound (and all such surfaces can be realizes as complex plane curves), the converse is not clear and is an interesting open question. However, realizing the slice-Bennequin bound is a property which satisfies the conditions of Theorem~\ref{brmon}, and so we have the following theorem.

\begin{thm} The set of braids in $B(n)$ whose underlying transverse link $L$ satisfies the bound $$sl(L) = -\chi_4(L)$$ forms a monoid $\Chi(n)$ in $B(n)$.
\end{thm}
 
\begin{question} Are the monoids $QP(n)$ and {$\Chi(n)$} in $B(n)$ the same?\end{question}

Additionally, there are many natural monoids coming from any invariant of an oriented link which are related to the slice-Bennequin inequality. Below we give the examples of the concordance invariants $s$ and $\tau$ in Khovanov homology and Heegaard Floer homology respectively.

We begin with the $s$ invariant.  In \cite{Rasmussen10}, Rasmussen used Khovanov homology to define an invariant $s(K)$ of knots in $S^3$ and proved that $s(K) \leq 2g_4(K)$. For our purposes, we want an invariant which gives an upper bound on the maximal self-linking number of a knot and Plamenevskaya \cite{Plamenevskaya06} proved that for transverse knots $sl(K) \leq s(K) - 1$. Thus $s(K)-1$ improves the slice-Bennequin bound on self-linking number. Additionally, we need our invariant to behave nicely under the addition of quasipositive half twists, and the lower bound on $-\chi$ of a cobordism provides this.   

In \cite{Pardon12}, Pardon extended the definition of Rasmussen's $s$ invariant from knots to links. 
For \emph{oriented} links $L$, one can extract from this a cobordism invariant which we will call $\tilde{d}(L)$. For knots, this agrees with $s(K) - 1$ and in general, $\tilde d$ behaves nicely under oriented cobordisms. (That is, this value of $\tilde d$ changes by no more than the Euler characteristic of the interpolating surface.) 
Additionally, Plamenevskaya's proof that $sl(K) \leq s(K) - 1$  for transverse knots extends without change to yield a bound $sl(L) \leq \tilde d (L)$. Thus, $\tilde d$ gives a condition which satisfies the hypotheses of Theorem~\ref{brmon}. 

\begin{thm} The set of braids in $B(n)$ whose underlying transverse link $L$ satisfies the bound $$ sl(L) = \tilde d(L)$$ forms a monoid $S(n)$ in $B(n)$.  
\end{thm}

Moreover, since $sl(L) \leq \tilde d(L) \leq -\chi_4(L)$, for each $n$ we have inclusions of monoids
$$QP(n) \subset \Chi(n) \subset S(n).$$

Now we turn to the $\tau$ invariant bound on self-linking number. Ozsv\'ath and Szab\'o \cite{OzsvathSzabo03} and Rasmussen \cite{Rasmussen03} introduced an invariant of knots in $S^3$ coming from the spectral sequence from knot Floer homology to the Floer homology of $S^3$ and show that this is a concordance invariant. Again, for convenience, $\tau$ was defined so that $\tau(K) \leq g_4(K)$ and so we define $\widetilde{\tau}(K) = 2 \tau(K) - 1$. There is a general belief (though currently no proof) that the following conjecture is true.
\begin{conj}
The invariant $\widetilde{\tau}$ can be extended to an invariant of oriented links (which behaves as expected with respect to oriented cobordisms). Additionally, $\widetilde \tau$ obeys the same self-linking number bound as it does for knots.
\end{conj}

\begin{mthm} The set of braids in $B(n)$ whose underlying oriented transverse link satisfies $$sl(L) = \widetilde\tau(L)$$ forms a monoid $T(n)$ in $B(n)$. 
\end{mthm}

Whenever $\tau$ does get extended to oriented links, it is expected that it  will still respect the slice genus bound $\widetilde{\tau} (L) \leq -\chi_4(L)$ and the self-linking number bound $sl(L) \leq \widetilde{\tau}(L)$ and so one would have an inclusions of monoids
$$QP(n) \subset \Chi(n) \subset T(n).$$

\begin{question} Assuming the conjecture above, are the monoids $S(n)$ and $T(n)$ distinct? There are examples of knots where $s$ and $2\tau$ do not agree \cite{HeddenOrding08}, which leaves open the possibility that they might indeed be distinct.\end{question}

\subsubsection{Non-vanishing of transverse invariants}
We now turn to Plamenevskaya's transverse invariant $\psi$ of transverse links in Khovanov homology. In \cite{Plamenevskaya06}, Plamenevskaya introduced a class $\psi(L)$ of a transverse link $L$ in the Khovanov Homology of the link $KH(L)$. The class is determined by a braid diagram $w$ representing a transverse link $L$. 
The invariant behaves functorially under the removal of quasi-positive half twists via crossing resolution (see Theorem~4 in \cite{Plamenevskaya06}) and so the non-vanishing of the invariant is preserved by the addition of quasi-positive half twists. The non-vanishing of the invariant is also clearly preserved under disjoint union. Thus by Theorem~\ref{braidmonoid} we have the following theorem. 

\begin{thm} 
The subset of the braid group $B(n)$ yielding transverse links whose corresponding $\psi$ invariant is nonzero forms a monoid in the braid group. We denote this monoid $\Psi(n)$.
\end{thm}

Plamenevskaya's work avoided the language of monoids but did show that all quasi-positive braids have nonzero $\psi$ invariant. In particular, the set of quasi-positive braids forms a a submonoid $$QP(n) \subset \Psi(n).$$

\begin{question} Is $\Chi(n) \subset \Psi(n)$? Are $S(n)$ and $\Psi(n)$ related?\end{question} 

\subsubsection{Cyclic branched covers} Given a transverse knot $T$ in any contact manifold $(M,\xi)$ then it is well known, and each to verify, that there is a contact structure $\widetilde{\xi}$ induced on any branched cover $\widetilde{M}$ of $M$ over $T$. We will consider transverse knots in $(S^3,\xi_{std})$ that come as the closers of braids $w$. Moreover, since the only covers one is always guaranteed to have are cyclic covers we restrict attention to these. We will denote the $n$--fold cyclic cover of $(S^3,\xi_{std})$ branched over the closure of the braid $w$ by $(S^3(w, n),\xi(w,n))$. 

One may easily check that if we are given two $k$--braids $w_1$ and $w_2$ and denote the result of stacking $w_2$ on top of $w_1$ by $w$ then $(S^3(w, n),\xi(w,n))$ is obtained from $(S^3(w_1, n),\xi(w_1,n))$ and $(S^3(w_2, n),\xi(w_2,n))$ by an $n$--fold connected sum (that is connected sum followed by $n-1$ internal connected sums). Moreover appending a quasi-positive half twist to a braid $w$ changes the $n$--fold branched cover by $n-1$ Legendrian surgeries. This is easily checked for double covers and verified in general in \cite{HarveyKawamuroPlamenevskaya09}.

This coupled with the discussion in Section~\ref{mvcg} implies that all the properties that give monoids in that section also give monoids in the braid group via branched coverings. For example the property defining the monoid $Tight(S)$ in $Map(S,\partial S)$ gives the monoid $Tight(n,k)$ in the braid group $B(n)$ consisting of braids whose closures give transverse links that induce tight contact structures on the $k$--fold cyclic cover. For each fixed $n$ and $k$ we have the same set of inclusions as discussed in Section~\ref{mvcg}. In addition \cite{HarveyKawamuroPlamenevskaya09} shows that 
\[
QP(n)\subset Dehn^+(n,k)
\]
for all $k$ and $n$. 
\begin{question} Is $QP(n)= Dehn^+(n,k)$ for any $n$ and $k$? \end{question}
The answer is almost certainly no, but for small $n$ and $k$ it might be true. (For example for $n=2$ and $k=2$ the answer is YES, but this might be the only such case with a positive answer.)

One may also easily see that for any property defining these monoids adding a strand preserves the property of the cover so we have, for example, $Tight(n,k)\subset Tight(n+1,k)$, and similarly for the other properties. 
\begin{question} Is there a relation between $Tight(n,k)$ and $Tight(n,k+1)$? What is the relation between $Tight(n,k)$ and $Tight(n+1,k)$? and similarly for the other branched cover monoids. \end{question}

There has been work relating $OzSz(n,2)$ to $\Psi(n)$, \cite{Baldwin11, BaldwinPlamenevskaya10}, but currently the exact relation is unknown. 
\begin{question} What is the relation between $OzSz(n,2)$ to $\Psi(n)$? Is $\Psi(n)\subset OzSz(n,2)$? \end{question}

\subsection{Co-products and monoids}
Given a homology theory that contains a transverse knot invariant and a co-product operation on the homology we can construct monoids as is illustrated below. We will see that these seem to be refinements of various monoids constructed above.

Recall that in \cite{OzsvathSzaboThurston08} Ozsv\'ath, Szab\'o and Thurston defined an invariant $\theta$ of transverse knots that lives in the (grid diagram formulation of) link Floer homology $HKL^-(m(L))$ of the mirror image of the link $L$. 
In \cite{Baldwin10}, Baldwin proved there is a comultiplication map on the link Floer homology of closures of braids
\[
\mu:HFL^-(m(\overline{w_1w_2})) \to HFL^-(m(\overline{w_1}\# \overline{w_1})),
\]
that respects the $\theta$ invariant. That is $\mu(\theta(\overline{w_1w_2}))=\theta(\overline{w_1}\# \overline{w_1})$. There are similar results for a stabilized version of the hat-invariant, $\widehat\theta$. It is known \cite{Vertesi08} that $\widehat\theta(\overline{w_1}\# \overline{w_1}$ is non-zero if and only if both $\widehat\theta(\overline{w_1})$ and $\widehat\theta(\overline{w_1})$ are both non-zero. This leads to the following result.

\begin{thm}[Baldwin 2010 \cite{Baldwin10}] Let $L_1$ and $L_2$ be transverse link with braid representatives $w_1$ and $w_2$ of the same braid index and define $L$ to be the transverse link represented by the braid $w_1 w_2$. If $\widehat\theta(L_1)$ and $\widehat\theta(L_2)$ are both non-zero, then $\widehat\theta(L)$ is also non-zero.\end{thm}

This obviously leads to a monoid by considering braids whose closures have non-zero $\widehat\theta$ invariants.

\begin{cor} The subset of the braid group $B(n)$ yielding transverse links whose corresponding $\widehat\theta$ invariant is non-zero forms a monoid in the braid group which we denote $\Theta(n)$.\end{cor} 

\begin{question} Are the monoids $\Psi(n)$ and $\Theta(n)$ distinct?\end{question}

\section{Studying monoids using contact geometry}\label{contact2studymonoid}

In this section we will show how some monoids in the braid group, that are not {\em a priori} related to contact geometry can be studied using using contact geometry. 

We start with a result about transverse knots.
\begin{thm}[Etnyre and Van Horn-Morris. 2010 \cite{EtnyreVanHornMorris10}]\label{uniquemax}
Let $K$ be a fibered knot in $S^3$ that is also the closure of a strongly quasi-positive braid and let $\Sigma$ be the associated Seifert surface (built as in Section~\ref{ht} form a strongly quasi-positive braid representing $K$). Then there is a unique transverse knot $T$ in the standard contact structure $\xi_{std}$ on $S^3$ that is topologically isotopic to $K$ and with $sl(T)=-\chi(\Sigma)$. 
\end{thm}

We now note several consequences of this result. First we state a ``quasi-positive recognition" result. 
\begin{cor}\label{cor1}
Let $K$ be a fibered, strongly quasi-positive knot in $S^3$ then a braid $b$ whose closure is $K$ is quasi-positive if and only if $a(b)=n(b)-\chi(K)$. (Here we denote by $a(b)$ the algebraic length of $b$, which was also called writhe earlier, and $n(b)$ is the braid index of $b$, that is the number of strands of $b$.)
\end{cor}
\begin{proof}
By hypothesis there is a word $w$ in the strongly quasi-positive generators of $SQP(n)$ such that its closure $\overline{w}$ is topologically isotopic to $K$. Moreover $\Sigma=\widetilde{\Sigma}_w$ (where $\widetilde{\Sigma}_w$ is the Seifert surface for $\overline{w}$ constructed in the Section~\ref{monoidsgs}) and one easily sees that
\[
\chi(\widetilde{\Sigma}_w)= n(w)-a(w)
\]
and of course as discussed above $\chi(K)=\chi(\overline{w})=\chi(\widetilde{\Sigma}_w)$. 
Also $\overline{w}$ is a transverse knot with self-linking $sl(\overline{w})=-\chi(\widetilde{\Sigma}_w)$ by Lemma~\ref{braidsl}. 

Now let $w'$ be some other braid word whose closure represents $K=\overline{w'}$. If we assume that 
\[
a(w')=n(b')-\chi(K)
\]
then we see that the transverse knot $\overline{w'}$ has self-linking $\sl(\overline{w'})=-\chi(\widetilde{\Sigma}_w)$. So by Theorem~\ref{uniquemax} we see that $\overline{w}$ and $\overline{w'}$ are transversely isotopic. Now  Theorem~\ref{posstab} says that $w$ and $w'$ are related by positive Markov moves and Theorem~\ref{qpandstab} then says that $w'$ must be quasi-positive. 

We now conversely assume that $w'$ is a quasi-positive braid representing $K$. 
Let $\Sigma'$ be the ribbon surface constructed from the quasi-positive braid $w'$ as in Section~\ref{monoidsgs}. As before we know that $\chi(\Sigma')=n(w')-a(w')$. Since positive stabilization preserves quasi-positivity and strong quasi-positivity we can assume that $n(w)=n(w')$. Moreover since all strongly quasi-positive generators are quasi-positive generators too we see that $a(w')\leq a(w)$. Now we see that
\[
sl(\overline{w'})=n(w')-a(w')\geq n(w)-a(w)=\chi(\Sigma).
\]
But Bennequin's inequality implies 
\[
sl(\overline{w'})=n(w')-a(w')\leq \chi(\Sigma)
\]
so we must have that  $a(w')= n(w')-\chi(K)$.
\end{proof}

We now turn to Orevkov's questions, specifically Questions~\ref{orevkov1} and~\ref{orevkov2}. 
\begin{cor}
Let $K$ be a fibered, strongly quasi-positive knot in $S^3$. Any two quasi-positive braids representing $K$ are related by positive Markov moves (and conjugation).
\end{cor}
\begin{remark}
In particular, this says that two positive braids represent the same knot if and only if they are related by positive Markov moves (and conjugation). So all questions about knots represented by positive braids can be answered purely in the Positive Braid Monoids. 
\end{remark}

\begin{remark}
The answer to Orevkov's Question~\ref{orevkov1} is NO for general  strongly quasi-positive knots as the following example shows (thus being fibered is a crucial hypothesis). In \cite{BirmanMenasco06II} Birman and Menasco showed that the two braids 
\[
\sigma_1^{2p+1}\sigma_2^{2r}\sigma_1^{2q}\sigma_2^{-1}
\]
and
\[
\sigma_1^{2p+1}\sigma_2^{-1}\sigma_1^{2q}\sigma_2^{2r}
\]
give the same topological knots but different transverse knots if $p+1\not=q\not=r$ and $p,q,r>1$. Note that they are of course strongly quasi-positive but since they are not transversely isotopic any sequence of Markov stabilizations taking one to the other must contain negative stabilizations. 
\end{remark}
\begin{proof}
Let $w_1$ and $w_2$ be two quasi-positive braids representing the same strongly quasi-positive fibered knot $K$. By Corollary~\ref{cor1} we know that their closures $\overline{w_1}$ and $\overline{w_2}$ are both transverse knots with self-linking number equal to $-\chi(K)$. Thus by Theorem~\ref{uniquemax} we know that $\overline{w_1}$ and $\overline{w_2}$ are transversely isotopic and so Theorem~\ref{posstab} implies they are related by positive Markov moves (and conjugation).
\end{proof}
For Orevkov's second questions we have the following partial answer. 
\begin{cor}
Let $K$ be a fibered, strongly quasi-positive knot in $S^3$.  Then any minimal braid index representative of $K$ is quasi-positive. 
\end{cor}
\begin{proof}
Dynnikov and Prasolov \cite{DynnikovPrasolov13} (see also LaFountain and Menasco \cite{MenascoLaFountain13?} for an alternate approach) proved the ``Kawamura Braid Geography Conjecture" (or also known as the ``generalized Jones conjecture"). Specifically they showed that given a knot $K$ if $b$ is the minimal braid index of braids representing $K$ then there is some constant $l$ such that for any braid $w$ representing $K$ we have
\[
b + |a(w)-l|\leq n(w).
\]
Graphically this is shown in Figure~\ref{bgc} where we plot all the values of $(a(w),n(w))$ for braids $w$ representing $K$. 
\begin{figure}[htb]
\begin{overpic}
{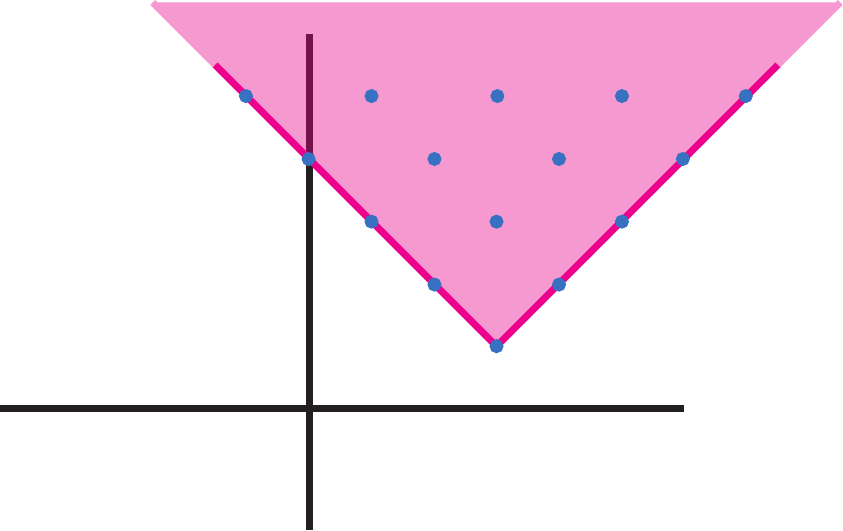}
\put(144, 45){$(l,b)$}
\put(201,33){$a(w)$}
\put(83, 146){$n(w)$}
\end{overpic}
\caption{The shaded region is the ``Braid Geography Cone" that contains the ordered pairs $(a(w),b(w))$ for all braids $w$ whose closures give the link $K$.}
\label{bgc}
\end{figure}

Now given a fibered, strongly quasi-positive knot $K$ we know from the Bennequin bound that for any braid $w$ representing $K$ we have
\[
a(w)-n(w)=sl(\overline{w})\leq -\chi(K).
\]
We know there is a strongly quasi-positive braid $w$ such that $K=\overline{w}$ and that $sl(\overline{w})= a(w)-n(w)=-\chi(K)$. So $(a(w),b(w))$ is on the right hand edge of ``Braid Geography Cone", see Figure~\ref{bgc}. Now if $w'$ is a minimal braid index braid representing $K$ then $(a(w'), n(w'))$ is at the vertex of the Cone and hence we see that $sl(\overline{w'})=a(w')-n(w')=-\chi(K)$. Thus Theorem~\ref{uniquemax} implies that $\overline{w}$ and $\overline{w'}$ are transversely isotopic and hence they are related by positive Markov moves (and conjugation) by Theorem~\ref{posstab}. Now of course Theorem~\ref{qpandstab} implies that $w'$ is quasi-positive. 
\end{proof}
\begin{remark}
It is still not known if the hypothesis of fibered is necessary in this result. In addition it is not known if strongly quasi-positive can be replaced with quasi-positive. 
\end{remark}

\bibliography{references}
\bibliographystyle{plain}

\end{document}